\documentclass[draft,reqno]{amsproc}

\usepackage{amsthm}
\usepackage{amsfonts}
\usepackage{amsmath}
\usepackage{amssymb}
\usepackage{mathrsfs}
\usepackage{epsfig}
\usepackage{graphicx}
\usepackage{epstopdf}
\usepackage{color}
\usepackage{ifthen}
\usepackage{float}
\usepackage[active]{srcltx}

\makeatletter
\@namedef{subjclassname@2010}{%
\textup{2010} Mathematics Subject Classification}
\makeatother

\newtheorem{theorem}{Theorem}[section]
\newtheorem{proposition}[theorem]{Proposition}%[section]
\newtheorem{corollary}[theorem]{Corollary}%[section]
\newtheorem{lemma}[theorem]{Lemma}%[section]

\theoremstyle{remark}
\newtheorem{remark}[theorem]{Remark}%[section]

\theoremstyle{definition}
\newtheorem{definition}[theorem]{Definition}%[section]
\newtheorem{example}[theorem]{Example}

\newcommand{\bq}{\begin{equation}}
\newcommand{\eq}{\end{equation}}
\newcommand{\beqn}{\begin{eqnarray*}}
\newcommand{\eeqn}{\end{eqnarray*}}
\newcommand{\beq}{\begin{eqnarray}}
\newcommand{\eeq}{\end{eqnarray}}

\newcommand{\rar}{\rightarrow}

\newcommand{\bc}{\begin{centre}}
\newcommand{\ec}{\end{centre}}

\newcommand{\ba}{\begin{array}}
\newcommand{\ea}{\end{array}}

\newcommand{\inp}[2]{\langle{#1},\,{#2} \rangle}

\renewcommand{\Delta}{{\nabla}}
\def \N{\mathbb{Z}_+}

\newcommand*{\child}[1]{\mathsf{Chi}(#1)}
\newcommand*{\childn}[2]{{\mathsf{Chi}}^{\langle#1\rangle}(#2)}
\newcommand*{\parentn}[2]{{\mathsf{par}}^{\langle#1\rangle}(#2)}

\newcommand*{\lambdab}{\boldsymbol\lambda}

\newcommand*{\parent}[1]{\mathsf{par}(#1)}
\newcommand*{\rot}{\mathsf{\omega}}

\begin{document}
\title[Weighted Shifts on Directed Trees]
{An Analytic Model for Left-Invertible \\ Weighted Shifts on Directed Trees}
   \author[S. Chavan]{Sameer Chavan}
   \address{Department of Mathematics and Statistics\\
Indian Institute of Technology Kanpur, India}
   \email{chavan@iitk.ac.in}
   \author[S. Trivedi]{Shailesh Trivedi}
   \address{School of Mathematics \\ Harish-Chandra Research Institute\\ 
Chhatnag Road, Jhu-nsi, Allahabad 211019, India}
\email{shaileshtrivedi@hri.res.in}
   
%   \thanks{The research of the
%third and fourth authors was supported by the NCN
%(National Science Center), decision No.
%DEC-2013/11/B/ST1/03613.}

\date{}

\maketitle

\begin{abstract}
Let $\mathscr T$ be a rooted directed tree with finite branching index $k_{\mathscr T}$ and 
let $S_{\lambda} \in B(l^2(V))$ be a left-invertible
weighted shift on ${\mathscr T}$. 
We show that
$S_{\lambda}$ can be modelled as a multiplication operator $\mathscr M_z$ on a reproducing kernel Hilbert space $\mathscr H$ of $E$-valued holomorphic functions on a disc centered at the origin, where $E:=\ker S^*_{\lambda}$.
The reproducing kernel associated with $\mathscr H$ is multi-diagonal and of bandwidth $k_{\mathscr T}.$
Moreover, $\mathscr H$ admits an orthonormal basis consisting of polynomials in $z$ with at most $k_{\mathscr T}+1$ non-zero coefficients.
As one of the applications of this model, we give
a complete spectral picture of $S_{\lambda}.$ 
Unlike the case $\dim E = 1,$ the approximate point spectrum of $S_{\lambda}$ could be disconnected. We also obtain an analytic model for left-invertible weighted shifts on rootless directed trees with finite branching index.  

\end{abstract}

%\part{Use this type of header for very long papers only}
% use lowercase except for proper names

%\section{Preparation and submission} % use lowercase except for proper names
%\label{intro}

\section{Preliminaries} 

The implementation of methods of graph theory into operator theory gives rise to a new class of operators known as {\it weighted shifts on directed trees}. These operators are generalization of adjacency operators of the directed trees. Although, the study of adjacency operators of the directed graphs was initiated by Fujii, Sasaoka and Watatani in \cite{F-S-W}, it was first observed by Jab{\l}o\'nski, Jung and Stochel in \cite{Jablonski} that replacing the directed graphs by directed trees not just gives a successful theory of weighted shifts but also provides a rich source of examples and counter-examples in operator theory \cite{JBS-1}, \cite{JBS-2}. 
Several questions related to boundedness, adjoints, normality, subnormality, hyponormality etc. of weighted shifts on directed trees have been studied in depth in \cite{Jablonski}. 

In the present paper, we discuss 
a rich interplay between the discrete structures (directed trees) and analytic structures (analytic kernels of finite bandwidth).
The starting point of this text is the observation that
any left-invertible weighted shift on a rooted directed tree can be realized as the operator of multiplication by the co-ordinate function on a reproducing kernel Hilbert space $\mathscr H$ of vector-valued holomorphic functions defined on a disc in the complex plane.
%These RKHS are in general structurally {\it different} from the Hilbert spaces obtained by tensoring a scalar-valued RKHS with a Hilbert space.
In case the directed tree has finite branching index, this analytic model takes a concrete form. In particular, the reproducing kernel associated with $\mathscr H$ turns out to be multi-diagonal. Also, the space $\mathscr H$ may not be obtained by tensoring a Hilbert space of scalar-valued holomorphic functions with another Hilbert space.
%This model can be effectively used to describe spectral picture of weighted shifts on rooted directed trees for instance. 
In this course, we arrive at a couple of interesting invariants, namely, branching index of a directed tree and radius of convergence for the weighted shift. Importantly, these invariants can be computed explicitly in various situations.

%Apart from the study of basic questions pertaining to boundedness, adjoints, normality, subnormality, hyponormality of linear operators, the class of weighted shifts on directed trees turns out to be a rich source of examples and counter-examples to several problems in operator theory.
%
%
%
%
% 
%
%
%However, the questions related to the spectrum of weighted shifts, realization of weighted shifts as some operator on other Hilbert space are yet to be answered. In view of this, the following question is a primary source of motivation of this paper.\\ 
%
%\noindent{\it Can a bounded weighted shift on a rooted directed tree be modelled as a multiplication operator on some reproducing kernel Hilbert space of holomorphic functions?}\\
% 
%We answer this question affirmatively to get an analytic model for left-invertible bounded weighted shifts on rooted directed trees, and in this course, we give a notion of {\it branching index} of a rooted directed tree. As an application of this model we describe the spectral theory of left-invertible bounded weighted shifts on rooted directed trees.

Let $\mathbb Z_+$, $\mathbb Z,$ $\mathbb R$ and $\mathbb C$ stand for
the sets of non-negative integers, integers, real numbers and complex
numbers, respectively. 
The complex conjugate of a complex number $w$ will be denoted by $\overline{w}.$
We use $\mathbb D_r$ 
to denote the open disc $\{z \in \mathbb C : |z| < r\}$ of radius $r > 0.$ 
In case $r=1,$ we denote the unit disc $\mathbb D_1$ by a simpler notation $\mathbb D.$  
%Denote by $\N$ and $\rbb_+$ the sets of nonnegative integers and nonnegative real numbers, respectively. 
%Given a set $X$, we denote by
%$\chi_{\varDelta}$ the characteristic function of a
%subset $\varDelta$ of $X$. 
For a subset $A$ of a non-empty set $X$, $\mbox{card}(A)$ denotes the cardinality of $A$.

Let $\mathcal H$ be a complex separable Hilbert space.
The inner-product  on $\mathcal H$ will be denoted by $\inp{\cdot}{\cdot}_{\mathcal H}$.
If no confusion is likely then we suppress the suffix, and simply write the inner-product as $\inp{\cdot}{\cdot}$.  
By a {\it subspace}, we mean a closed linear manifold.
Let $W$ be a subset of $\mathcal H.$ Then $\mbox{span}\,W$ stands for the smallest linear manifold generated by $W.$
In case $W$ is singleton $\{w\},$ we use the convenient notation $\langle w \rangle$ in place of  $\mbox{span}\,\{w\}$.
By $\bigvee \{w : w \in W\},$ 
we understand the subspace generated by $W$.
For a subspace $\mathcal M$ of $\mathcal H,$ we use $P_{\mathcal M}$ to denote the orthogonal
projection of $\mathcal H$ onto $\mathcal M.$ For vectors $x, y \in \mathcal H,$ we use the notation $x \otimes y$ to denote the rank one operator given by
 $$x \otimes y (h) = \inp{h}{y}x,~h \in \mathcal H.$$ 

Unless stated otherwise, all the Hilbert spaces occurring below are complex 
infinite-dimensional separable and for any such Hilbert space 
$\mathcal H$, ${B}({\mathcal H})$ denotes the Banach algebra of bounded linear operators on $\mathcal H.$ 
For $T \in B(\mathcal H),$ the symbols $\ker T$ and $\mbox{ran}\,T$ will stand for the kernel and 
the range of $T$ respectively.
The Hilbert space adjoint of $T$ will be denoted by $T^*.$
In what follows, we denote the spectrum, approximate point spectrum, essential spectrum and the point spectrum
%the approximate point spectrum, the right spectrum, the essential spectrum of
of $T$ by ${\sigma(T)}$, ${\sigma_{ap}(T)}$, ${\sigma_e(T)}$ and 
${\sigma_p(T)}$ 
%${\sigma_{ap}(S)},$ $\sigma_r(S),$ $\sigma_e(T)$
respectively. 
We reserve the notation $r(T)$ for the spectral radius of $T.$

Let $T \in B(\mathcal{H})$.
We say that $T$ is {\it left-invertible} if there exists $S \in B(\mathcal{H})$ such that $ST=I.$ Note that $T$ is left-invertible if and only if
there exists a constant $\alpha > 0$ such that $T^*T \geq \alpha I.$ In this case, $T^*T$ is invertible and $T$ admits the left-inverse $(T^*T)^{-1}T^*$.
Following \cite{Shimorin}, we refer to the operator $T'$ given by $T':=T(T^*T)^{-1}$ as the {\it Cauchy dual} of the left-invertible operator $T.$
Further, we say that $T$ is {\it analytic} if $\bigcap_{n \geq 0}T^n(\mathcal H)=\{0\}$. If $\mathscr H$ is a reproducing kernel Hilbert space of holomorphic functions defined on a disc in $\mathbb C$, then the multiplication operator $\mathscr M_z$ defined on $\mathscr H$ provides an example of an analytic operator. It is interesting to note that 
almost all analytic operators arise in this way.
Indeed, a result of S. Shimorin \cite{Shimorin} asserts that any left-invertible analytic operator is unitarily equivalent to the operator of multiplication by $z$ on a reproducing kernel Hilbert space of vector-valued holomorphic functions defined on a disc.
Since the proof of this fact, as given in \cite[Sections 1 and 2]{Shimorin}, plays a major role in the proof of the main result, we outline it in the following discussion
(cf. \cite[Theorem 2.13]{SV}). 
%We also need the notion of wandering subspace. Recall that $T$ admits the {\it wandering subspace property} if 
%$${\mathcal H} = {{\bigvee}_{n\geq0}}T^n (\ker T^*).$$ The {\it wandering subspace} $\ker T^*$ has the following property: $$T^n(\ker T^*) \perp \ker T^* \mbox{~for every integer}~n \geq 1.$$

Let $T \in B(\mathcal{H})$ be a left-invertible analytic operator
and
let $E:=\ker T^*$. For each $x \in \mathcal H$, define an $E$-valued holomorphic function $U_x$ as 
$$U_x(z)=\displaystyle \sum_{n \geq 0}(P_ET'^{*n}x)z^n,$$
where $T'$ is the {Cauchy dual} of $T$. 
A simple application of the spectral radius formula \cite{Conway} shows that the function $U_x(z)=P_E(I-zT'^*)^{-1}x$ is well-defined and holomorphic on the disc $\mathbb D_{r}$, where $r:=\frac{1}{r(T')}.$ Let $\mathscr H$ denote the vector space of $E$-valued holomorphic functions of the form $U_x$, $x \in \mathcal H$. Consider
the map $U:\mathcal H \rar \mathscr H$ defined by $Ux = U_x$. By \cite[Lemma 2.2]{Shimorin}, the kernel of $U$ is precisely
$\bigcap_{n \geq 0}T^n(\mathcal H)$, and hence by the assumption, $U$
is injective.
%We reproduce a lemma from \cite{Shimorin} for ready reference.
%\begin{lemma}\cite[Lemma 2.2]{Shimorin}
%The map $U:\mathcal H \rar \mathscr H$ is injective if and only if $T$ is analytic.
%\end{lemma}
In particular, we may equip the space $\mathscr H$ with the norm induced from $\mathcal H$, so that $U$ is unitary. It turns out that $\mathscr H$ is a $z$-invariant reproducing kernel Hilbert space with $UT=\mathscr M_zU$, where $\mathscr M_z$ is the operator of multiplication by $z$. Also, the reproducing kernel $\kappa_{\mathscr H} :\mathbb D_{r} \times \mathbb D_{r} \rar {B}(E)$ is given by
\beq \label{rk} \kappa_{\mathscr H}(z,w)=\displaystyle\sum_{j,k\geq0}P_ET'^{*j}T'^k|_Ez^j\overline{w}^k, \eeq
which satisfies the following:
\begin{itemize}
\item[(i)] for any $x \in E$ and $\lambda \in \mathbb D_{r},$
$$\kappa_{\mathscr H}(\cdot,\lambda)x \in \mathscr H;$$
\item[(ii)] for any $x \in E$, $h \in \mathscr H$ and $\lambda \in \mathbb D_{r},$
$$\langle h(\lambda),x \rangle_E=\langle h,\kappa_{\mathscr H}(\cdot,\lambda)x \rangle_{\mathscr H}.$$
\end{itemize}
Conditions (i) and (ii) may be rephrased by saying that the set of bounded point evaluations (for short, bpe) for $\mathscr H$ contains the disc $\mathbb D_r.$
We see in the context of weighted shifts on rooted directed trees that indeed (analytic) bpe contains the disc $\mathbb D_{r_{\lambda}}$ of larger radius $r_{\lambda}$ (see Definition \ref{defwn}). This occupies the major part of the proof of the main result. 

In the remaining part of this section, we invoke some basic concepts from the theory of directed trees which will be frequently used in the rest of this paper. The reader is referred to \cite{Jablonski} for a detailed exposition on directed trees.

A pair $\mathscr T= (V,\mathcal E)$ is called a {\it directed graph} if $V$ is a non-empty set and $\mathcal E$ is a subset of $V \times V \setminus \{(v,v): v \in V\}$. An element of $V$ (resp. $\mathcal E$) is called a {\it vertex} (resp. an {\it edge}) of $\mathscr T$.  A finite sequence $\{v_i\}_{i=1}^n$ of distinct vertices is said to be a {\it circuit} of $\mathscr T$ if $n \geq 2$, $(v_i,v_{i+1}) \in \mathcal E$ for all $1 \leq i \leq n-1$ and $(v_n,v_1) \in \mathcal E$. A directed graph $\mathscr T$ is said to be {\it connected} if for any two distinct vertices $u$ and $v$ of $\mathscr T$, there exists a finite sequence $\{v_i\}_{i=1}^n$ of vertices of $\mathscr T$ $(n \geq 2)$ such that $u=v_1$, $v_n=v$ and $(v_i,v_{i+1})$ or $(v_{i+1},v_i) \in \mathcal E$ for all $1 \leq i \leq n-1$. For a
subset $W$ of $V$, define $\child{W} = \bigcup_{u\in W} \{v\in V
\colon (u,v) \in \mathcal E\}.$ 
One may define inductively $\childn{n}{W}$ for 
$n\in \N$ as follows: Set
$\childn{n}{W}=W$ if $n=0$, and
$\childn{n}{W}=\child{\childn{n-1}{W}}$ if $n\geqslant
1$.
%Let
%$$
%\des{W}=\bigcup_{n=0}^{\infty} \childn{n}{W}.$$
Given $v\in V$, we write $\child{v}:=\child{\{v\}}$,
$\childn{n}{v}=\childn{n}{\{v\}}$. A member of $\child{v}$ is called a {\it child} of $v.$ For a given vertex $v\in V$, if there exists a unique vertex
$u \in V$ such that $(u,v)\in \mathcal E$, we say that $v$ has a {\em parent} $u$ and denote it by $\parent{v}$. A vertex $v$ of $\mathscr T$ is called a {\it root} of $\mathscr T$, or $v \in \mathsf{Root}(\mathscr T)$, if there is no vertex $u$ of $\mathscr T$ such that $(u,v)$ is an edge of $\mathscr T$. If $\mathsf{Root}(\mathscr T)$ is a singleton then its unique element is denoted  by $\mathsf{root}$. We set $V^\circ:=V \setminus \mathsf{Root}(\mathscr T)$.

%\begin{definition}
A directed graph $\mathscr T= (V,\mathcal E)$ is called a {\it directed tree} if 
\begin{itemize}
\item[(i)] $\mathscr T$ has no circuits,
\item[(ii)] $\mathscr T$ is connected and
\item[(iii)] each vertex $v \in V^\circ$ has a parent.
\end{itemize}
%\end{definition}
\begin{remark}
Any directed tree has at most one root \cite[Proposition 2.1.1]{Jablonski}. 
\end{remark}
A directed tree $\mathscr T$ is said to be
\begin{enumerate}
\item[(i)] {\it rooted} if it has a (unique) root.
\item[(ii)] {\it rootless} if it has no root.
\item[(iii)]
{\it locally finite} if $\mbox{card}(\child u)$ is finite for all $u \in V.$ 
\item[(iv)]
{\it leafless} if every vertex has at least one child.
\end{enumerate}
 
%The directed trees $\mathscr T$ considered in this text are {\it not} necessarily locally finite.

%In what follows, we always assume that $S_{\lambda}$ is injective. In this case, $\mathscr T$ is necessarily leafless 
%\cite[Proposition 3.1.7]{Jablonski}. 

%For a subset $W$ of $V,$ let $\mbox{Des}(W)$ denote the set of
%descendents of $W.$ 

%\begin{remark}
%The kernel condition for 
%\end{remark}
In what follows, $l^2(V)$ stands for the Hilbert
space of square summable complex functions on $V$
equipped with the standard inner product. Note that
the set $\{e_u\}_{u\in V}$ is an
orthonormal basis of $l^2(V)$, where $e_u \in l^2(V)$
is the indicator function $\chi_{\{u\}}$ of $\{u\}$. Given a system
$\lambdab = \{\lambda_v\}_{v\in V^{\circ}}$ of non-negative real numbers, 
we define the {\em weighted shift operator} $S_{\lambda}$ on ${\mathscr T}$
with weights $\lambdab$ by
   \begin{align*}
   \begin{aligned}
{\mathscr D}(S_{\lambda}) & := \{f \in l^2(V) \colon
\varLambda_{\mathscr T} f \in l^2(V)\},
   \\
S_{\lambda} f & := \varLambda_{\mathscr T} f, \quad f \in {\mathscr
D}(S_{\lambda}),
   \end{aligned}
   \end{align*}
where $\varLambda_{\mathscr T}$ is the mapping defined on
complex functions $f$ on $V$ by
   \begin{align*}
(\varLambda_{\mathscr T} f) (v) :=
   \begin{cases}
\lambda_v \cdot f\big(\parent v\big) & \text{if } v\in
V^\circ,
   \\
   0 & \text{if } v \text{ is a root of } {\mathscr T}.
   \end{cases}
   \end{align*}
Unless stated otherwise, $\{\lambda_v\}_{v\in V^{\circ}}$ consists of positive numbers and $S_{\lambda}$ belongs to $B(l^2(V)).$ It may be concluded from \cite[Proposition 3.1.7]{Jablonski} that $S_\lambda$ is an injective weighted shift on $\mathscr T$ if and only if $\mathscr T$ is leafless. 
In what follows, we always assume that all the directed trees considered in this text are countably infinite and leafless. 
%In view of this, $S_{\lambda}$ on a leafless directed tree is left-invertible if and only if it has closed range.

In the proof of the main result, we frequently use the following elementary facts pertaining to the weighted shifts on directed trees.
\begin{lemma} \label{facts} 
If $S_{\lambda} \in B(l^2(V)),$
then
for any $u \in V$ and positive integer $k,$
\begin{enumerate}
\item[(i)] 
$  S^k_{\lambda}e_{u} = \displaystyle \sum_{v \in \childn{k}{u}} \lambda_v \lambda_{\parent v}\cdots \lambda_{\parentn{k-1}{v}} e_v$ and \\
%$S_{\lambda}e_u = \sum_{v \in \mathsf{Chi}(u)}\lambda_v e_v$.
%\item[(ii)] 
$ \|S^k_{\lambda}e_{u}\|^2 = \displaystyle \sum_{v \in \childn{k}{u}} \big(\lambda_v \lambda_{\parent v}\cdots \lambda_{\parentn{k-1}{v}} \big)^2.$
%$\|S_{\lambda}e_u\|^2 = \sum_{v \in \mathsf{Chi}(u)}\lambda_v^2$.
\item[(ii)] 
$\displaystyle S^{*k}_{\lambda}e_{u} =  \lambda_u \lambda_{\parent u}\cdots \lambda_{\parentn{k-1}{u}} e_{\parentn{k}{u}}$ and \\ 
%$S^*_{\lambda}e_u = {\lambda}_ue_{\mathsf{par}(u)}$ if $u$ is not a root, and $S^*_{\lambda}e_u=0$ otherwise.
%\item[(iv)] 
$\displaystyle \|S^{*k}_{\lambda}e_{u}\|^2 =  \big(\lambda_u \lambda_{\parent u}\cdots \lambda_{\parentn{k-1}{u}} \big)^2,$
where $e_{{\mathsf{par}}^{\langle 
n \rangle}(v)}$ is understood to be the zero vector  in case ${\mathsf{par}}^{\langle 
n \rangle}(v)=\emptyset.$ 
\item[(iii)]
$S^{*k}_{\lambda}S^k_{\lambda}e_u = \|S^k_{\lambda}e_u\|^2e_u$.
\end{enumerate} 
\end{lemma}
\begin{proof}
The part (i) has been established in \cite[Lemma 6.1.1]{Jablonski}, whereas (ii) and (iii) can be obtained by a straightforward mathematical induction using \cite[Lemma 3.4.1(iii)]{Jablonski}.
\end{proof}

Let $S_\lambda$ be a left-invertible weighted shift on a rooted 
directed 
tree with weights 
$\{\lambda_v\}_{v\in V^{\circ}}$. It can be easily seen from (i) and (iii) above that the Cauchy dual $S'_{\lambda}$ of 
$S_{\lambda}$ is given by 
\[S'_{\lambda}e_u:=\displaystyle\sum_{v\in\mathsf{Chi}(u)}\frac{\lambda_v}{
\|S_\lambda 
e_{\mathsf{par}(v)}\|^2}e_v~ \text{for all}\ v\in V^{\circ}.\]
Note that $S'_{\lambda} \in B(l^2(V))$ is a weighted shift with weights 
$\{\lambda'_v\}_{v\in V^{\circ}},$ where 
\beq \label{weights}
\lambda'_v:=\frac{\lambda_v}{\|S_\lambda 
e_{\mathsf{par}(v)}\|^2}~ \text{for all}\ v\in V^{\circ}.
\eeq
This also shows that $\{\lambda'_v\}_{v\in V^{\circ}}$ is a bounded subset of positive real line.
Throughout this text, we find it convenient to use the notation $S_{\lambda'}$ in place of $S'_{\lambda}$.

It turns out that any weighted shift $S_{\lambda}$ on a rooted directed tree $\mathscr T$ 
is analytic (see Lemma \ref{lem1} below).
Hence by Shimorin's construction as described above, any left-invertible $S_{\lambda}$ admits an analytic model $(\mathscr M_z, \kappa_{\mathscr H}, \mathscr H)$. It turns out that this model can be significantly improved upon provided the underlying directed tree has  finite branching index (see Definition \ref{b-index}). In this case, the analytic model takes a concrete form with {\it multi-diagonal} kernel $\kappa_{\mathscr H}$ defined on a disc $\mathbb D_{r_{\lambda}}$, where $r_{\lambda}$
is a positive number such that $\frac{1}{r(S_{\lambda'})} \leq r_{\lambda} \leq r(S_{\lambda})$ (see \eqref{radius}). Moreover,
the reproducing kernel Hilbert space admits an orthonormal basis consisting of vector-valued analytic polynomials.
One of the interesting aspects of our model is a handy formula for $r_{\lambda}$ depending on $\mathscr T$ and $S_{\lambda}.$

Although the motivation for the present work comes mainly from the theory of weighted shifts on directed trees as expounded in \cite{Jablonski}, it is closely related to some of the recent developments in the function theoretic operator theory. In particular, the reader is referred to the study of analytic reproducing kernels of finite bandwidth carried out in a series of papers by G. Adams et al \cite{AM}, \cite{AM-1}, \cite{AFM}, \cite{AMSS} (refer also to \cite{Ar} for the general theory of reproducing kernels).  
It is also worth noting that the class of weighted shifts on rooted directed trees has some resemblance with the class of adjoints of abstract weighted shifts (in the context of complex Hilbert spaces) \cite{B-1}, \cite{B-2}, \cite{R-1} and also with the class of 
operator-valued weighted shifts \cite{L}, \cite{K}, \cite{Ja-0}, \cite{Si}
studied extensively in the literature. 
%However, it is not hard to construct directed tree $\mathscr T$ so that
%$S^*_{\lambda}$ on $\mathscr T$ is not an abstract weighted shift in the sense of \cite{R-1}.  
 
Here is the sketch of the paper. Section 2 is devoted to the statement of the main theorem and some of its immediate consequences. The proof of main theorem is presented in Section 3. In Section 4, we present several examples illustrating the rich interplay between the directed trees and reproducing kernels of finite bandwidth. 
In Section 5, we use the main theorem to describe various spectral parts of $S_{\lambda}.$ It turns out that weighted shifts on directed trees with disconnected approximate point spectra are in abundance. 
In the final section, we introduce a notion of branching index
for rootless directed trees and use it to obtain an analytic model for a left-invertible weighted shift $S_{\lambda}$ in this setting. 
It turns out that $S_{\lambda}$ is an extension of a weighted shift operator on a rooted directed tree.

\section{Main Result: Statement and Consequences}

%The proof of the main result relies heavily on the basic theory of 
%weighted shifts on directed trees as expounded in \cite{Jablonski}.

Let $\mathscr T=(V, \mathcal E)$ be a rooted directed tree with root $\mathsf{root}$. Then
\beq \label{disjoint}
V = \bigsqcup_{n = 0}^{\infty} \childn{n}{\mathsf{root}}~(\mbox{disjoint union})
\eeq
(see \cite[Proposition 2.1.2]{Jablonski}).
%We set $n_{\mathsf{root}}:=0$ and for
For each $u\in V$, let $n_u$ denote the unique non-negative integer 
such that
$u \in \mathsf{Chi}^{\langle n_u\rangle}(\mathsf{root})$. We use the convention that 
$\mathsf{Chi}^{\langle{j}\rangle}(\mathsf{root})=\emptyset$ if $j < 0.$ Similar convention holds for $\mathsf{par}$.

The statement of the main theorem involves an invariant (to be referred to as the branching index) associated with a rooted directed tree. 

\begin{definition} \label{b-index}
Let $\mathscr T$ be a rooted directed tree  
and let
$$V_{\prec}:=\{u\in V: \mbox{card}(\mathsf{Chi}(u)) \geq2\}$$ be the set of branching 
vertices of $\mathscr T$. Define
\[k_\mathscr{T}:=\begin{cases}
 1+\sup\{n_w:w\in V_{\prec}\},& \text{if $V_{\prec}$ is non-empty}\\
 0,& \text{if $V_{\prec}$ is empty}.
\end{cases}
\]
We refer to $k_{\mathscr T} \in \mathbb Z_+ \cup \{\infty\}$ as the {\it branching index} of $\mathscr 
T.$ 
\end{definition}
\begin{remark}
If $\mbox{card}(V_{\prec})$ is finite then so is $k_{\mathscr T}$. On the other hand, directed trees $\mathscr T$ with infinite $\mbox{card}(V_{\prec})$ and finite $k_{\mathscr T}$ can be constructed easily. 
\end{remark}

%The branching index is closely related to the multiplicity of weighted shift $S_{\lambda}$. 

The condition (i) in the following proposition says precisely that $\mathscr T$ is Fredholm (refer to \cite[Section 3.6]{Jablonski} for more details related to Fredholm directed trees). 

\begin{proposition} \label{probranch}
Let $S_\lambda \in B(l^2(V))$ be a weighted shift on a rooted directed tree $\mathscr T$ with root
$\mathsf{root}$. 
Let
$V_{\prec}$ be the set of branching 
vertices of $\mathscr T$ and let $k_{\mathscr T}$ be the branching index of $\mathscr T.$ 
Then the following statements are equivalent:
\begin{enumerate}
 \item [(i)] $\mathscr T$ is locally finite such that $\mbox{card}(V_{\prec})$ is finite. 
 \item [(ii)] $\mathscr T$ is locally finite such that $k_{\mathscr T}$ is finite.
 \item [(iii)] The dimension of $E:=\ker S_\lambda^*$ is finite.
\end{enumerate}
\end{proposition}
\begin{proof}
That (i) implies (ii) is obvious. Suppose that (ii) holds. If $V_\prec$ is not finite, then that $k_\mathscr T$ is finite implies 
that there exists an infinite subset $W$ of $V_\prec$ such that $n_w$ is 
constant for all $w \in W$. Clearly, $W \subseteq \mathsf {Chi}^{\langle 
n_w \rangle} (\mathsf {root})$. Therefore, there exists a vertex $v \in V$ 
with $n_v < n_w$ such that $\mbox{card}(\mathsf{Chi}(v))$ is infinite. This contradicts the assumption that $\mathscr T$ is locally 
finite. Thus (ii) implies (i). 

By \cite[Proposition 
3.5.1(ii)]{Jablonski},   
\beqn 
E=\ker S_{\lambda}^*=\langle e_{\mathsf{root}}\rangle \oplus \bigoplus_{v \in 
V}\left(l^2(\mathsf{Chi}(v)) \ominus \langle \lambdab^v\rangle \right),
\eeqn
where
$\lambdab^v : \mathsf{Chi}(v) \rar \mathbb C$ is defined by 
$\lambdab^v(u)=\lambda_u,$ and $\langle f \rangle$ denotes the span of 
$\{f\}$. 
Observe now that $l^2(\mathsf{Chi}(v)) \ominus \langle\lambdab^v\rangle \neq \{0\}$ if and only if $v \in V_{\prec}$. Therefore, 
\beq \label{kernel} 
E=\langle e_{\mathsf{root}}\rangle \oplus \bigoplus_{v \in 
V_{\prec}}\left(l^2(\mathsf{Chi}(v)) \ominus \langle \lambdab^v\rangle \right).
\eeq
It now follows from \eqref{kernel} that $\dim E$ is finite if and only if $\mbox{card}(\child v)$ is finite for every $v \in V_{\prec}$ and $\mbox{card}(V_{\prec})$ is finite.
This gives the equivalence of (i) and (iii).
\end{proof}
\begin{remark}
It may happen that $k_{\mathscr T} < \infty$ and $\dim E = \infty$ (see Example \ref{tree-inf}).
\end{remark}

%Let $T$ be a left-invertible operator in 
%$B(\mathcal H).$ Then {\it Cauchy dual} of $T$ is given by $T':=T(T^*T)^{-1}$ 
%\cite{Shimorin}.
%Note that 
%if $S_{\lambda}$ is left-invertible then $S_{\lambda'}$ coincides with 
%Shimorin's notion $S'_{\lambda}:=S_{\lambda}(S^*_{\lambda}S_{\lambda})^{-1}$ 
%of Cauchy dual as introduced and studied in \cite{Shimorin}.

\begin{definition}\label{defwn}
Let $\mathscr T$ be a rooted directed tree with 
root $\mathsf{root}$ and let $k_{\mathscr T}$ be the branching index of $\mathscr T.$
For any integer $n$, consider the set 
\beqn 
W_n:= \bigcup_{j=n}^{k_{\mathscr T} + n} 
\mathsf{Chi}^{\langle{j}\rangle}(\mathsf{root}).
\eeqn
Let $S_\lambda \in B(l^2(V))$ be a left-invertible
weighted shift with weights $\{\lambda_v\}_{v\in V^{\circ}}$ 
and let $S_{\lambda'}$ be the Cauchy dual of $S_{\lambda}.$
%Suppose that $E:=\ker S_\lambda^*$ is of finite dimension. Set
The {\it radius of convergence} 
for $S_{\lambda}$ is defined as the non-negative number $r_{\lambda}$ given by
\beq \label{radius}
r_{\lambda}:=\liminf_{n \rar \infty} \left(\sum_{v \in W_n} \big(\lambda'_v 
\lambda'_{\mathsf{par}(v)} \cdots 
\lambda'_{{\mathsf{par}}^{\langle n-1 
\rangle}(v)}\big)^2 \right)^{-\frac{1}{2n}}.
\eeq
\end{definition}

We will see later that $r_{\lambda}$ is positive whenever $k_{\mathscr T}$ is finite (see Lemma \ref{r-lambda}). Let us compute $r_{\lambda}$ in the case in which $S_{\lambda}$ is a unilateral weighted shift.

\begin{example}[(Diagonal)] \label{diagonal} 
Consider the directed tree $\mathscr T_1$ with the set of vertices 
$V:=\mathbb{Z}_+$ and $\mathsf{root}=0$. We further require that $\mathsf{Chi}(n)=\{n+1\}$ for
all 
$n \geq 0$. For future reference, we note that 
$V_\prec=\emptyset$, and hence $k_{\mathscr
T_1}=0$.
The weighted shift $S_\lambda$ on the directed tree $\mathscr T_1$ (to be referred to as {\it unilateral weighted shift})
is given by 
\[S_\lambda e_n=\lambda_{n+1}e_{n+1}\ \text{for all}\ n\geq0.\]
(Caution: This differs from the standard definition $S_\lambda e_n=\lambda_{n}e_{n+1}\ \text{for all}\ n\geq0$ of the unilateral weighted shift.)
%Assume that $\{\lambda_n\}_{n \geq 0}$ is a bounded subset of the 
%interval $[a, \infty)$ for a positive real number $a.$ Thus $S_{\lambda}$ is 
%left-invertible and $E :=\ker S_\lambda^*=\{\alpha e_0 : \alpha \in \mathbb
%C\}$. 
It is well-known that $S_{\lambda}$ is unitarily equivalent
to the operator ${\mathscr M}_z$ of multiplication by $z$ on the reproducing kernel Hilbert space $\mathscr H$ associated with the kernel 
\[\kappa_{\mathscr 
H}(z,w)=1 +\displaystyle\sum_{j\geq1}C_{j,j}z^j\overline{w}^j~(z, 
w \in \mathbb D_{r}),\]
where $r:=\liminf_{n \rar \infty} \left({\lambda_n} 
{\lambda_{n-1}} \cdots 
{\lambda_{1}}\right)^{\frac{1}{n}}$ and $\{C_{j, j}\}_{j \geq 0}$ is a sequence of positive numbers
(refer to \cite{S}).
Since $W_n=\{n\}$ for every $n \in \mathbb Z_+,$ the radius of convergence $r_{\lambda}$  for (a left-invertible) $S_{\lambda}$ is precisely
$r.$ 
Moreover, one can verify that (the rank one operator) 
$C_{j,j}$ is (multiplication by) $\frac{1}{\lambda^2_1 
\cdots\lambda^2_{j}}$
for all $j\geq1$. 
Clearly, the reproducing kernel $\kappa_{\mathscr H}(\cdot,\cdot)$ is
{\it diagonal} in this case.
\end{example}

 We are now in a position to state the main result of this paper.

\begin{theorem} \label{thm1}
%Let $\mathscr T=(V, \mathscr E)$ be a rooted directed tree with root 
%$\mathsf{root}$. Assume that the branching index $k_{\mathscr T}$ of $\mathscr T$ is finite.
Let $\mathscr T$ be a rooted directed tree with finite branching index $k_{\mathscr T}.$ 
Let $S_\lambda \in B(l^2(V))$ be a left-invertible weighted shift and let $S_{\lambda'}$ be the Cauchy
dual of $S_{\lambda}.$ 
Set $E:= \ker S^*_{\lambda}$.
Then there exist a $z$-invariant reproducing kernel Hilbert space 
$\mathscr H$ of $E$-valued holomorphic functions defined on the disc
$\mathbb{D}_{r_\lambda}$ and a unitary mapping $U:l^2(V)\longrightarrow\mathscr H$ such
that ${\mathscr M}_zU=US_\lambda,$ where ${\mathscr M}_z$ 
denotes the operator of multiplication by z on $\mathscr H$ and $r_{\lambda}$ is the radius of convergence for $S_{\lambda}$. Moreover, $r_{\lambda} r(S_{\lambda'}) \geq 1,$ where $r(S_{\lambda'})$ is the spectral radius of $S_{\lambda'}.$
Further, $U$ maps $E$ onto the subspace $\mathscr E$
of $E$-valued constant functions in $\mathscr H$ such that $Ug=g$ for every $g \in E.$ 
Furthermore,
we have the following:
\begin{itemize}
\item [(i)] The reproducing kernel $\kappa_{\mathscr H} : \mathbb D_{r_{\lambda}} 
\times \mathbb D_{r_{\lambda}} \rar B(E)$ associated with 
$\mathscr H$ satisfies $\kappa_{\mathscr H}(\cdot,w)g \in \mathscr H$ and
$
\inp{Uf}{\kappa_{\mathscr H}(\cdot,w)g}_{\mathscr H} = \inp{(Uf)(w)}{g}_E$
for every $f \in l^2(V)$ and $g \in E.$
\item [(ii)] $\kappa_{\mathscr H}$ is given by
 \begin{equation}\label{eq2}
\kappa_{\mathscr H}(z,w)=I_E+\displaystyle\sum_{\underset{|j-k|\leq
k_\mathscr{T}}{j,k\geq1}}C_{j,k}z^j\overline{w}^k~(z, w \in \mathbb D_{r_{\lambda}}),
 \end{equation}
 where $I_E$ denotes the identity operator on $E$, and $C_{j,k}$ are bounded 
linear operators on $E$ given by 
 $$C_{j,k}=P_ES^{*j}_{\lambda'}S_{\lambda'}^k|_E~(j, k =1, 2, \cdots)$$ with 
$P_E$ 
being the orthogonal projection of $l^2(V)$ onto $E$. 
 \item [(iii)] The $E$-valued polynomials in $z$ are dense in $\mathscr H$. In fact,
$$\mathscr H=\bigvee\{z^nf:f\in {\mathscr E},\ n\geq0\}.$$
\item [(iv)] $\mathscr H$ admits an orthonormal basis consisting of polynomials 
in $z$ with at most $k_\mathscr{T}+1$ non-zero coefficients.
 \end{itemize}
\end{theorem}
\begin{remark}
Let $S_\lambda \in B(l^2(V))$ be a left-invertible weighted shift with non-negative
weights $\{\lambda_v\}_{v\in V^{\circ}}$. 
%Assume that the branching index $k_{\mathscr T}$ of $\mathscr T$ is finite. 
Let $u_0 \in V$ be such that $\parent {u_0} \in V_{\prec}$.
Suppose that $\lambda_{u_0} =0$ and $\lambda_{u} > 0$ for all $u \in V \setminus \{u_0\}$. Then by \cite[Proposition 3.1.6]{Jablonski}, $S_{\lambda}$ can be decomposed as an orthogonal direct sum of two weighted shifts $S_{\lambda,1}, S_{\lambda, 2}$ on directed trees with positive weights. Since $S_{\lambda}$ is left-invertible, so are $S_{\lambda, 1}$ and $S_{\lambda, 2}$. 
%Also, the dimensions of cokernels of $S_1$ and $S_2$ are finite. 
By the theorem above, there exist multiplication operators $\mathscr M_z^{(i)}$ on reproducing kernel Hilbert spaces
$\mathscr H^{(i)}$ for $i=1, 2$ such that $S_{\lambda}$ is unitarily equivalent to
${\mathscr M_z^{(1)}} \oplus {\mathscr M_z^{(2)}}$. Note that $\mathscr H_1 \oplus \mathscr H_2$ is the reproducing kernel Hilbert space associated with the reproducing kernel $\kappa_{\mathscr H^{(1)}} + \kappa_{\mathscr H^{(2)}}$ (refer to \cite{Ar}).
\end{remark}
%\begin{remark}
%%Note that $r_{\lambda} = r = \frac{1}{r(S_{\lambda'})}$ in case $E$ is 
%%one-dimensional (see Example \ref{diagonal} below). 
%%One of essential aspects of Theorem \ref{thm1} is that in Shimorin's model the radius $\frac{1}{r(S_{\lambda'})}$ of domain of convergence can be replaced by the maximum of $\frac{1}{r(S_{\lambda'})}$ and $r_{\lambda}.$ We remark that there is no known handy formula for ${r(S_{\lambda})}.$ On the other hand, the variant $r_{\lambda}$ can be easily computed in practice (refer to Section 3).      
%\end{remark}

The  inequality $r_{\lambda} r(S_{\lambda'}) \geq 1$ in Theorem \ref{thm1} may be strict in general. 

\begin{example}
Consider the weighted shift $S_\lambda$ on the directed tree $\mathscr T_1$ (as discussed in Example \ref{diagonal}) 
with weights $\lambda_n$ given by  
$$\lambda_1=\lambda_2=\lambda_3=\lambda_4=1,\ \mbox{and}\ \lambda_k=\begin{cases}
\frac{1}{2}, & \mbox{if}\ 2^n+1 \leq k \leq 3.2^{n-1},\ n \geq 2\\
1, & \mbox{otherwise.}
\end{cases}$$
Note that $\inf_{n\geq1}\lambda_n=\frac{1}{2}$. Thus $S_\lambda$ is left-invertible and hence $S_{\lambda'}$ is bounded. Further, for any $n \geq 1,$ total number of $\frac{1}{2}$'s occurring in first $2^n$ places is equal to $2^{n-1}-2^{n-2}+2^{n-2}-2^{n-3}+\cdots+4-2=2^{n-1}-2$. Therefore, we get 
$$\lambda_1 \lambda_2 \cdots \lambda_{2^n}=\frac{1}{2^{2^{n-1}-2}}=\frac{2^2}{2^{2^{n-1}}}.$$ 
Let $n$ be any positive integer. Then there is a unique positive integer $m_n$ such that $2^{m_n} \leq n < 2^{m_n+1}$. Let $n=2^{m_n}+k$ for some integer $k$ such that $0 \leq k < 2^{m_n}$.  Therefore, 
$$\lambda_1 \lambda_2 \cdots \lambda_{n} \geq \frac{1}{2^{2^{m_n-1}-2+k}}=\frac{2^2}{2^{2^{m_n-1}+k}}.$$ 
Since $k < 2^{m_n}$, $\frac{2^{m_n -1} + k}{n} = 1- \frac{2^{m_n -1}}{2^{m_n}+k} < \frac{3}{4}.$ It follows that
\beqn \left(\lambda_1 \lambda_2 \cdots \lambda_{n}\right)^{\frac{1}{n}} \geq \left(\frac{2^2}{2^{2^{m_n-1}+k}}\right)^{\frac{1}{n}} > 2^{\frac{2}{n} - \frac{3}{4}}, \eeqn and hence $r_{\lambda} = \liminf_{n \rar \infty} \left({\lambda_n} 
{\lambda_{n-1}} \cdots 
{\lambda_{1}}\right)^{\frac{1}{n}}$ is at least $2^{-\frac{3}{4}}.$
On the other hand, $r(S_{\lambda'})=2.$ This can be seen as follows.
Note that $\lambda'_n=\frac{1}{\lambda_n}$. Therefore, \beqn r(S_{\lambda'})=\lim_{n \rar \infty}(\sup_{m\geq1}\lambda'_{m+1}\cdots\lambda'_{m+n})^\frac{1}{n}=(2^n)^\frac{1}{n}=2 \eeqn (since ${2}$'s occur in $\{\lambda'_n\}$ consecutively at $2^{n-1}$ places for $n \geq 2$). Thus we have $r_{\lambda} r(S_{\lambda'}) \geq 2^{\frac{1}{4}},$ which is obviously bigger than $1.$
\end{example}

Since the proof of Theorem \ref{thm1} consists of several observations of independent interest, it will be presented in the next section. 
In the remaining part of this section, we discuss some immediate consequences of the main theorem. 
First a terminology.

Let ${\mathscr M}_z, \kappa_{\mathscr H}$, and $\mathscr H$ be as appearing in 
the statement of 
Theorem \ref{thm1}.
For the sake of convenience, we will refer to the triple $({\mathscr M}_z, 
\kappa_{\mathscr H}, 
\mathscr H)$ as the {\it analytic model} of the left-invertible weighted shift $S_{\lambda}$ 
acting on the directed tree $\mathscr T.$

{\it Except the final section of this paper, we assume that $\mathscr T$ is a leafless, rooted directed tree with finite branching index $k_{\mathscr T}.$}

An operator 
$T$ in ${B}({\mathcal H})$ is said to be {\it finitely cyclic} if there are a finite number of vectors $h_1, \cdots, h_m$
in $\mathcal H$ such that 
\beqn
{\mathcal H} = \bigvee {\{T^k h_i : k \geq 0, i=1, \cdots, m \}}.
\eeqn
In case $m=1$, we refer to $T$ as {\it cyclic operator} with {\it cyclic vector} $h_1.$
We say that $T$ is {\it infinitely cyclic} if it is not finitely cyclic.

%The following is immediate from Theorem \ref{thm1}(iii).
\begin{corollary} \label{cyclic}
Let $S_\lambda \in B(l^2(V))$ be a weighted shift on $\mathscr T$.
If $E:= \ker S^*_{\lambda}$ is finite dimensional then $S_\lambda$ is finitely cyclic.
\end{corollary}
\begin{proof}
Since $\mathscr T$ is leafless, by 
\cite[Proposition 3.1.7]{Jablonski}, $S_{\lambda}$ is injective.
If $E:= \ker S^*_{\lambda}$ is finite dimensional then the range of $S_{\lambda}$ is closed, and hence 
$S_{\lambda}$ is left-invertible.
Now appeal to Theorem \ref{thm1}(iii).
\end{proof}

In general, the reproducing kernel Hilbert space $\mathscr H$ as constructed in the proof of Theorem \ref{thm1} can not be realized as the tensor product $\mathscr K \otimes E$, where $\mathscr K$ is a Hilbert space of scalar-valued holomorphic functions. 
We make this explicit in the following result. 

\begin{corollary} \label{tensor} 
%Let $\mathscr T$ be a rooted directed tree with finite branching index $k_{\mathscr T}.$
Let $S_\lambda \in B(l^2(V))$ be a left-invertible weighted shift on  
$\mathscr T$. Let $({\mathscr M}_z, 
\kappa_{\mathscr H}, 
\mathscr H)$ denote the analytic model of $S_\lambda$ and let $E:= \ker S^*_{\lambda}$.
Suppose that there exist a Hilbert space $\mathscr K$ of scalar-valued holomorphic functions and an isometric isomorphism $\Phi : \mathscr H \rar \mathscr K \otimes E$ such that $\Phi(p f) = p \otimes f$ for every polynomial $p \in \mathscr K$ and $f \in E$.
Then the reproducing kernel $\kappa_{\mathscr H}$ associated with $\mathscr H$ is the diagonal kernel given by
\beqn
\kappa_{\mathscr H}(z,w)=I_E+\displaystyle\sum_{j=1}^{\infty}
\left(P_ES^{*j}_{\lambda'}S_{\lambda'}^j|_E\right) z^j\overline{w}^j~(z, w \in \mathbb D_{r_{\lambda}}).
\eeqn
\end{corollary}
\begin{proof} Note that for any $f, g \in E$ and $m, n \in \mathbb Z_+$, 
\beq \label{orthogonal}
\inp{S^m_{\lambda}f}{S^n_{\lambda} g}_{l^2(V)}=\inp{z^mf}{z^n g}_{\mathscr H} = \inp{\Phi(z^mf)}{\Phi(z^n g)}_{\mathscr K \otimes E}=
\inp{z^m}{z^n}_{\mathscr K}\inp{f}{g}_{E}.
\eeq
Since $S^k_{\lambda}e_{\mathsf{root}} \in \bigvee \{e_v : v \in \childn{k}{\mathsf{root}}\}$, by an application of \eqref{disjoint},
we obtain $\inp{z^m}{z^n}_{\mathscr K}=0$ for $m \neq n$ after letting $f=e_{\mathsf{root}}=g.$ 
Hence by \eqref{orthogonal}, we must have $\inp{S^m_{\lambda}f}{S^n_{\lambda} g}_{l^2(V)}=0$
for any $f, g \in E$ and non-negative integers $m \neq n.$
This shows that the sequence $\{S^k_{\lambda}E\}_{k \geq 0}$ of subspaces of $l^2(V)$ is mutually orthogonal. 
It follows immediately that for any $f, g \in E$ and non-negative integers $j \neq k,$ \beqn \inp{P_ES^{*j}_{\lambda'}S_{\lambda'}^kf}{g}_E= \inp{S^{*j}_{\lambda'}S_{\lambda'}^kf}{g}_{l^2(V)}=\inp{S_{\lambda'}^kf}{S^{j}_{\lambda'}g}_{l^2(V)}=0. \eeqn
In particular, $P_ES^{*j}_{\lambda'}S_{\lambda'}^k|_E=0$ for all non-negative integers $j \neq k.$ The desired conclusion now follows from Theorem \ref{thm1}(ii).
\end{proof}
\begin{remark}
In view of Shimorin's model (as discussed in Section 1), after replacing $r_{\lambda}$ by $\frac{1}{r(S_{\lambda'})}$, one may obtain the conclusion of Corollary \ref{tensor} for any directed tree with infinite branching index $k_{\mathscr T}$. Thus, even for directed trees $\mathscr T$ with infinite $k_{\mathscr T}$, the associated reproducing kernel $k_{\mathscr H}$ could be multi-diagonal.
%In the spirit of above corollary, $S_{\lambda}$ may not be unitarily equivalent to orthogonal direct sum of unilateral weighted shifts (see Remark \ref{direct-sum}).
\end{remark}

%\begin{remark}
%In general, if $\dim E$ is bigger than $1$, the reproducing kernel Hilbert space $\mathscr H$ as constructed in the proof of Theorem \ref{thm1} can not be realized as the tensor product $\mathscr K \otimes E$, where $\mathscr K$ is a Hilbert space of scalar-valued holomorphic functions. To see this, suppose that $\Phi : \mathscr H \rar \mathscr K \otimes E$ is an isometric isomorphism such that $\Phi(pf) = p \otimes f$ for every polynomial $p \in \mathscr K$ and $f \in E.$ 
%\end{remark}

%The analytic model of $S_{\lambda}$ can be used efficiently to compute the point spectrum of $S_{\lambda}$, which is otherwise difficult. 

Recall that $T \in B(\mathcal H)$ is an {\it isometry} if $T^*T=I.$
\begin{corollary} 
%Let $\mathscr T$ be a rooted directed tree with finite branching index $k_{\mathscr T}.$
Consider the analytic model $({\mathscr M}_z, 
\kappa_{\mathscr H}, 
\mathscr H)$
 of a left-invertible weighted shift $S_{\lambda}$ on $\mathscr T$ and let $E:= \ker S^*_{\lambda}$. If $S_{\lambda}$ is an isometry then $\kappa_{\mathscr H}$ is the $B(E)$-valued Cauchy kernel given by 
 \beqn
 \kappa_{\mathscr H}(z, w) = \frac{I_E}{1- z \overline{w}}~(z, w \in \mathbb D).
 \eeqn
 In particular, $\mathscr H$ is the $E$-valued Hardy space of the open unit disc.
\end{corollary}
\begin{proof}
Assume that $S_{\lambda}$ is an isometry.
Note that $S_{\lambda'}$ is also isometry in view of hypothesis and $S^*_{\lambda'}S_{\lambda'} =(S^*_{\lambda}S_{\lambda})^{-1}.$ By the uniqueness of the reproducing kernel, it suffices to check that $C_{j, k}=\delta_{j, k}I_E$, where $$C_{j,k}:=P_ES^{*j}_{\lambda'}S_{\lambda'}^k|_E ~(j, k =1, 2, \cdots)$$ and $\delta_{j, k}$ denotes the Kronecker delta. If $j=k$ then
obviously $C_{j, k} = I_E.$ If $j < k$ then $C_{j, k} =
P_ES_{\lambda'}^{k-j}|_E = 0$ since $S_{\lambda'}E \subseteq \mbox{ran}\, S_\lambda = E^{\perp}.$   
\end{proof}
%\begin{remark}
One rather striking consequence of the preceding corollary is as follows:
{\it If $S_{\lambda}$ is an isometry then $r_{\lambda}=1$}. Note that this observation is irrespective of the structure of the directed tree $\mathscr T$. On the other hand, the definition of $r_{\lambda}$ relies on the weight sequence $\{\lambda_v\}_{v \in V^{\circ}}$ and of course on the structure of $\mathscr T.$ To see this fact, assume that $S_{\lambda}$ is an isometry.
By Theorem
\ref{thm1}, we must have $r_{\lambda} r(S_{\lambda'}) \geq 1.$ However,
$S_{\lambda'}$ being isometry, $r(S_{\lambda'})=1$, and hence $r_{\lambda} \geq 1.$
Since $\frac{I_E}{1- z \overline{w}}$ is not defined on $\mathbb D_r \times \mathbb D_r$ for any $r > 1,$ by Theorem \ref{thm1}(ii), $r_{\lambda}$ can not exceed $1.$ 
%\end{remark}

\section{Proof of the Main Theorem}

The proof of Theorem \ref{thm1} involves several lemmas. The first of which collects some facts related to the set $W_n$. Recall from Definition \ref{defwn} that for any integer $n$, the set $W_n$ is given by 
\beq \label{Wn}
W_n:= \bigcup_{j=n}^{k_{\mathscr T} + n} 
\mathsf{Chi}^{\langle{j}\rangle}(\mathsf{root}).
\eeq

\begin{lemma} \label{branch}
%Let $\mathscr T=(V, \mathscr E)$ be a rooted directed tree with root
%$\mathsf{root}$ and let $k_\mathscr T$ be the branching index of $\mathscr T$.
 Let
$S_\lambda \in B(l^2(V))$ be a weighted shift on $\mathscr T$ and let $E:= \ker S^*_{\lambda}$. 
%If $E:= \ker S^*_{\lambda}$ is finite dimensional 
Then we have the following statements:
\begin{enumerate}
 \item[(i)] $E$ is a subspace of the (possibly infinite dimensional) space $\bigvee \{e_v: v \in W_0\}$.
 \item[(ii)] $\mbox{card}(\mathsf{Chi}^{\langle n \rangle 
}(\mathsf{root}))$ (possibly countably infinite) is constant for $n\geq k_\mathscr T$. In particular, $\mbox{card}\,(W_n)$ is constant for $n \geq k_{\mathscr T}$.
 \item[(iii)] For 
every $v 
\notin W_n$, $e_{{\mathsf{par}}^{\langle 
n \rangle}(v)}$ belongs to the orthogonal complement of $E,$
where $e_{{\mathsf{par}}^{\langle 
n \rangle}(v)}$ is understood to be the zero vector  in case ${\mathsf{par}}^{\langle 
n \rangle}(v)=\emptyset.$ 
\item[(iv)] For non-negative integers $m$ and $n,$ $W_n \cap W_m \neq \emptyset$ if and only if $|n-m| \leq k_{\mathscr T}.$ 
\end{enumerate}
\end{lemma}

\begin{proof}
%Since $E$ is finite dimensional, by Proposition \ref{probranch}, we get that  $\mathscr T$ is locally finite such that
%$k_\mathscr T$ is finite. 
Note that $\child{V_{\prec}} \subseteq W_0.$ Hence by \eqref{kernel}, 
\beq \label{ker-W0} E \subseteq \langle e_{\mathsf {root}}\rangle 
\oplus \bigoplus_{v \in 
V_{\prec}} l^2(\mathsf{Chi}(v)) \subseteq \bigvee \{e_v : v \in W_0\}.\eeq
This yields (i). To see (ii), recall that $n_u$ is the unique non-negative integer 
such that $u \in \mathsf{Chi}^{\langle n_u\rangle}(\mathsf{root})$. 
Note that $\mbox{card}(\mathsf{Chi}(u))=1$ if 
$n_u\geq k_\mathscr T$, where we used the assumption that $\mathscr T$ is leafless.
Thus $\mbox{card}(\mathsf{Chi}^{\langle n \rangle 
}(\mathsf{root}))$ is constant for $n\geq k_\mathscr T$. 
This proves (ii).

We now check (iii).
Let $v \notin W_n.$ Since  
$E$ is a subspace of $\bigvee \{e_v : v \in W_0\}$ by part 
(i),  it suffices to check that $e_{{\mathsf{par}}^{\langle 
n \rangle}(v)}$ is orthogonal to $\{e_v : v 
\in W_0\}$. Note that $n_v < n$ or $n_v > k_{\mathscr T} + n,$ 
If $n_v < n$ then $e_{{\mathsf{par}}^{\langle 
n \rangle}(v)}=e_{\emptyset} =0$ by convention. Otherwise, 
${\mathsf{par}}^{\langle 
n \rangle}(v) \notin W_0,$ and hence $e_{{\mathsf{par}}^{\langle 
n \rangle}(v)}$ is orthogonal to $\{e_v : v 
\in W_0\}$. To see (iv), let $n, m$ be two non-negative integers such that $n < m.$ If $n + k_{\mathscr T} < m$ then
clearly, $W_n \cap W_m = \emptyset.$ So suppose that $n + k_{\mathscr T} \geq m.$
Then $W_n \cap W_m = \cup_{k=m}^{n+k_{\mathscr T}} \childn{k}{\mathsf {root}},$ which is obviously non-empty.
\end{proof}

Next we prove that the radius of convergence for any left-invertible weighted shift with finite dimensional cokernel is positive.
\begin{lemma} \label{r-lambda}
Let $S_\lambda \in B(l^2(V))$ be a left-invertible
weighted shift on $\mathscr T$ and
let $S_{\lambda'}$ be the Cauchy dual of $S_{\lambda}.$
%Suppose that $E:=\ker S_\lambda^*$ is of finite dimension. 
%Set 
%\beq \label{radius}
%r_{\lambda}:=\liminf_{n \rar \infty} \left(\sum_{v \in W_n} \lambda'_v 
%\lambda'_{\mathsf{par}(v)} \cdots 
%\lambda'_{{\mathsf{par}}^{\langle n-1 
%\rangle}(v)}\right)^{-\frac{1}{n}},
%\eeq
%where $W_n$ is given by \eqref{Wn}. 
If $r(S_{\lambda'})$ denotes the spectral radius of $S_{\lambda'}$ then the radius of convergence $r_{\lambda}$ for $S_{\lambda}$ satisfies $r_{\lambda} r(S_{\lambda'}) \geq 1.$ In particular, $r_{\lambda}$ is 
positive.
\end{lemma}

\begin{proof}
By Lemma \ref{facts}(i), for any integer $k \geq 0,$
$$\|S^k_{\lambda'}e_{\mathsf {root}}\|^2 = \sum_{v \in \childn{k}{\mathsf {root}}} \big(\lambda'_v \lambda'_{\parent v}\cdots \lambda'_{\parentn{k-1}{v}} \big)^2.$$
It follows from \eqref{disjoint} that 
\beqn
\sum_{v \in W_n}\big(\lambda'_v \lambda'_{\parent v}\cdots \lambda'_{\parentn{k-1}{v}} \big)^2 &=& \sum_{k=n}^{n+k_{\mathscr T}} 
\sum_{v \in \childn{k}{\mathsf {root}}} \big(\lambda'_v \lambda'_{\parent v}\cdots \lambda'_{\parentn{k-1}{v}} \big)^2 \\ &=&
\sum_{k=n}^{n+k_{\mathscr T}} \|S^k_{\lambda'}e_{\mathsf {root}}\|^2 
= \sum_{k=0}^{k_{\mathscr T}} \|S^{n+k}_{\lambda'}e_{\mathsf {root}}\|^2 \\
&\leq & \|S^{n}_{\lambda'}\|^2 \sum_{k=0}^{k_{\mathscr T}} \|S^k_{\lambda'}e_{\mathsf {root}}\|^2.
\eeqn
If we set $M:=\sum_{k=0}^{k_{\mathscr T}}\|S^k_{\lambda'}e_{\mathsf {root}}\|^2$ (which is finite since $k_{\mathscr T} < \infty$) then
by the definition of $r_{\lambda},$
\beqn
r_{\lambda} \geq  \liminf_{n \rar \infty} \Big(M\|S^{n}_{\lambda'}\|^2 \Big)^{-\frac{1}{2n}}  =
\Big(\lim_{n \rar \infty} M^{-\frac{1}{2n}} \Big) \Big(\liminf_{n \rar \infty} \|S^{n}_{\lambda'}\|^{-\frac{1}{n}}\Big) = \frac{1}{r(S_{\lambda'})},
\eeqn 
which completes the proof of the lemma.
\end{proof}
\begin{remark}
If $S_{\lambda}$ is an expansion (that is, $S^*_{\lambda} S_{\lambda} \geq I)$ then $S_{\lambda'}$ is a contraction (that is, $S^*_{\lambda} S_{\lambda} \leq I)$. In this case, $r_{\lambda}$ is at least $1.$
\end{remark}

We need one more fact in the proof of the main result (cf. \cite[Proposition 4.5]{ACJS}).
\begin{lemma}\label{lem1}
Let $S_\lambda \in B(l^2(V))$ be a 
weighted shift on $\mathscr T$. Then $S_\lambda$ is analytic.
\end{lemma}

\begin{proof}
Put $V_0:=V\ \text{and}\ V_k:=\displaystyle V \setminus 
\cup_{j=0}^{k-1}\mathsf{Chi}^{\langle j\rangle}(\mathsf{root})~(k \geq 1).$ 
Note that $\{V_k\}_{k\geq0}$ is a strictly decreasing 
sequence of sets such that $\displaystyle \cap_{k\geq0}V_k=\emptyset$. Now, 
for all $u\in V$ and all integers $k\geq0$, by Lemma \ref{facts}(i),
\[S_\lambda^k e_u=\displaystyle 
\sum_{v \in \childn{k}{u}} \lambda_v \lambda_{\parent v}\cdots \lambda_{\parentn{k-1}{v}} 
e_v.\]
It follows that
\[\mbox{ran}\,S_\lambda^k \subseteq\bigvee\{e_u:u\in V_k\}:=M_k,\ \text{say}.\]
Also, if $f\in M_k$, then $f(u)=0$ for $u\in V\setminus V_k = 
\cup_{j=0}^{k-1}\mathsf{Chi}^{\langle j\rangle}(\mathsf{root})$. Thus, if
$f \in \displaystyle\cap_{k=0}^{\infty} M_k$, then $f(u)=0$ for
$u \in \displaystyle \cup_{j=0}^{\infty} \mathsf{Chi}^{\langle 
j\rangle}(\mathsf{root}) =V$. That is, $f=0$. Hence
\beqn \{0\} \subseteq \displaystyle \cap_{k = 0}^{\infty}\, \mbox{ran}\, 
S_\lambda^k \subseteq \cap_{k=0}^{\infty}\, M_k =\{0\}. \eeqn This shows that 
$S_\lambda$ is
analytic.
\end{proof}

\begin{proof}[{Proof of Theorem \ref{thm1}}] 
As mentioned earlier, the proof relies on the 
ideas developed in \cite[Sections 1 and 
2]{Shimorin}. 
%Note that $\ker S^*_{\lambda}=E=\ker S^*_{\lambda'}.$ 
Let
$f = \sum_{v \in V}f(v)e_v 
\in l^2(V).$ By Lemmas \ref{facts}(ii) and \ref{branch}(iii), 
\beqn
P_ES_{\lambda'}^{*n}f 
&=& \sum_{v \in V} f(v)\lambda'_v 
\lambda'_{\mathsf{par}(v)} \cdots 
\lambda'_{{\mathsf{par}}^{\langle n-1 \rangle}(v)}P_Ee_{{\mathsf{par}}^{\langle 
n \rangle}(v)} \\
&=& \sum_{v \in W_n} f(v)\lambda'_v 
\lambda'_{\mathsf{par}(v)} \cdots 
\lambda'_{{\mathsf{par}}^{\langle n-1 \rangle}(v)}P_Ee_{{\mathsf{par}}^{\langle 
n \rangle}(v)}.
\eeqn
We claim that the $E$-valued series 
\beq \label{Uf}
U_f(z):=\displaystyle \sum_{n\geq0}(P_ES_{\lambda'}^{*n}f)z^n \eeq converges 
absolutely in $E$ on the disc $\mathbb D_{r_{\lambda}}$ for every $f \in l^2(V)$.
%Note that for any integer $k \geq 0,$
%\beqn
%\Big\|\sum_{n = 0}^k (P_ES_{\lambda'}^{*n}f)z^n \Big\| \leq \|f\| \sum_{n \geq 
%0} \|S^n_{\lambda'}\||z|^n,
%\eeqn
%and hence
%by the spectral radius formula, $U_f(z)$ converges absolutely on the disc of 
%radius $\frac{1}{r(S_{\lambda'})}.$ Let us check that $U_f(z)$ converges 
%absolutely on the disc of 
%radius $r_{\lambda}.$
By Lemma \ref{branch}(iv), for non-negative integers $m$ and $n,$ $W_n \cap W_m \neq \emptyset$ if and only if $|n-m| \leq k_{\mathscr T}.$ 
It follows that
\beq \label{estimate}
\sum_{\underset{n \geq 0}{v \in W_n}} |f(v)|^2 &=& \sum_{v \in W_0} |f(v)|^2 + \sum_{v \in W_1} |f(v)|^2 + \cdots \nonumber \\ &\leq & 
\sum_{v \in V} (k_{\mathscr T}+1)|f(v)|^2  = (k_{\mathscr T}+1)\|f\|^2.
\eeq
%By Lemma \ref{branch}, $W_n$ is a finite set 
%for every integer $n \geq 0,$ and for $n \geq k_{\mathscr T},$ 
%the 
%cardinality $|W_n|$ of $W_n$ is constant. In particular, $M:=\sup_{n \geq 
%0}|W_n| < 
%\infty.$ 
Now
by the Cauchy-Schwarz inequality, for any integer $k \geq 0,$
\beqn
\Big\|\sum_{n = 0}^k (P_ES_{\lambda'}^{*n}f)z^n \Big\|
&\leq & \sum_{\underset{n \geq 0}{v \in W_n}} |f(v)|\lambda'_v 
\lambda'_{\mathsf{par}(v)} \cdots 
\lambda'_{{\mathsf{par}}^{\langle n-1 \rangle}(v)} 
|z|^n \\ &\leq & \Big(\sum_{\underset{n \geq 0}{v \in W_n}} |f(v)|^2 
\Big)^{\frac{1}{2}} \Big( 
\sum_{\underset{n \geq 0}{v \in W_n}} \big(\lambda'_v 
\lambda'_{\mathsf{par}(v)} \cdots 
\lambda'_{{\mathsf{par}}^{\langle n-1 \rangle}(v)}\big)^2|z|^{2n}\Big)^{\frac{1}{2}} 
\\ 
&\overset{\eqref{estimate}} \leq & \sqrt{k_{\mathscr T}+1}~\|f\|\Big( 
\sum_{\underset{n \geq 0}{v \in W_n}} \big(\lambda'_v 
\lambda'_{\mathsf{par}(v)} \cdots 
\lambda'_{{\mathsf{par}}^{\langle n-1 \rangle}(v)}\big)^2|z|^{2n}\Big)^{\frac{1}{2}}.
%&\leq & \sqrt{k_{\mathscr T}+1}~\|f\|\Big( 
%\sum_{\underset{n \geq 0}{v \in W_n}} \lambda'_v 
%\lambda'_{\mathsf{par}(v)} \cdots 
%\lambda'_{{\mathsf{par}}^{\langle n-1 \rangle}(v)}|z|^{n}\Big). 
\eeqn
Since the series on the right hand side converges absolutely on $\mathbb D_{r_{\lambda}}$, 
the claim stands verified.
Thus $U_f$
is holomorphic in the disk $\mathbb{D}_{r_{\lambda}}.$ This allows us to define the map $U : l^2(V) \rar \mathscr H$ by $Uf = U_f,$ where 
$\mathscr H$ denotes the complex vector space of $E$-valued holomorphic 
functions of the form $U_f$. 
By Lemma \ref{lem1}, $S_{\lambda}$ is analytic, and hence by \cite[Lemma 2.2]{Shimorin}, $U$ is injective. Thus the inner-product given by
\[\langle U_f,U_g\rangle=\langle f,g\rangle_{l^2(V)}\ \text{for all}\ f,g\in 
l^2(V)\]
makes $\mathscr H$ an inner-product space. Also, the very definition of the inner-product on 
$\mathscr H$ shows that $U$ is unitary, and hence $\mathscr H$ is a Hilbert 
space. 

Note that for each $f\in E$, 
$U_f(z)=f$. 
%Here, by abuse of notation, we denote the element $f$ in $E$ as well as the constant $E$-valued function with value $f$ by the same notation. Thus $U|_E=I$. 
We now show that $\mathscr H$ is $z$-invariant. Let 
$U_f\in\mathscr H$. Since $S^*_{\lambda'}S_{\lambda} = I$ and
$P_ES_{\lambda}=0,$ we get
\beqn zU_f(z) 
&=& \displaystyle\sum_{n\geq0}(P_ES_{\lambda'}^{*n}f)z^{n+1} =\sum_{n\geq1} 
(P_ES_{\lambda'}^{*n-1}f)z^n \\ &=& \sum_{n\geq0}(P_ES_{\lambda'}^{*n}S_\lambda 
f)z^n=U_{S_\lambda f}(z)\in\mathscr H.\eeqn
Above expression also verifies that ${\mathscr M}_zU=US_\lambda$, where 
${\mathscr M}_z$ is the 
operator of multiplication by z on $\mathscr H$. 

Part (i) has been recorded on \cite[Pg 154]{Shimorin} (see the discussion following \eqref{rk}).
To see (ii), recall from \eqref{rk} that  
\[\kappa_{\mathscr 
H}(z,w)=\displaystyle\sum_{j,k\geq0}C_{j,k}z^j\overline{w}^k,\]
where $C_{j,k}$ is a bounded linear operator on $E$ given by 
$C_{j,k}=P_ES_{\lambda'}^{*j}S_{\lambda'}^k|_E$. Since ${\ker}
S_{\lambda'}^*={\ker} 
S_\lambda^*=E$, it follows that $P_ES_{\lambda'}^{*j}|_E=0$ for all $j\geq1$. 
%Again, since
%$(\text{ran}\ S_{\lambda'})^\bot = E$, therefore $S_{\lambda'}^jE$ is orthogonal to $E$ for all $j 
%\geq 1$. 
Since $C^*_{j, k} = C_{k, j}$, we get
$C_{j,0}=0=C_{0,j}$ for all $j\geq1$. Hence the above expression for 
$\kappa_{\mathscr H}(z,w)$ reduces to 
\[\kappa_{\mathscr 
H}(z,w)=I_E+\displaystyle\sum_{j,k\geq1}C_{j,k}z^j\overline{w}^k.\]
As recorded earlier in \eqref{ker-W0}, $E \subseteq \bigvee \{e_v : v \in W_0\}$ (see also \eqref{Wn} for the definition of $W_n$).
It follows that  
\[S_{\lambda'}^{*j}S_{\lambda'}^kE\subseteq \bigvee\big\{
e_v:v \in W_{k-j}\big\},\] and
therefore 
$S_{\lambda'}^{*j}S_{\lambda'}^kE$ is orthogonal to $E$ if 
$|j-k|>k_\mathscr T$. Thus $C_{j,k}=0$ 
if 
$|j-k|>k_\mathscr T$.  This proves (ii). 

To prove (iii), note that by Lemma \ref{lem1}, $S_{\lambda'}$ is 
analytic. Therefore, by \cite[Proposition 2.7]{Shimorin}, 
\beqn
l^2(V) = \bigvee_{n \geq 0} S^n_{\lambda}(E).
\eeqn
%$S_\lambda$ admits 
%the wandering subspace property with $E$ as a wandering subspace. 
Since ${\mathscr M}_z$ is 
unitarily equivalent to 
$S_\lambda$ and $\ker S^*_{\lambda}=E$, it follows that 
\beqn
\mathscr H = \bigvee_{n \geq 0} {\mathscr M}^n_{z}(\mathscr E).
\eeqn
This is precisely (iii).  

Finally, since $U$ is unitary and 
$\{e_u:u\in V\}$ is an orthonormal basis of $l^2(V)$, $\{U_{e_u}:u\in V\}$ is 
an 
orthonormal basis of $\mathscr H$. Note that
\beq \label{onb} U_{e_u}(z)=\displaystyle\sum_{k\geq0}(P_ES_{\lambda'}^{*k}e_u)z^k=\sum_{
0\leq k\leq n_u}(P_ES_{\lambda'}^{*k}e_u)z^k\eeq
(see the discussion following \eqref{disjoint} for the definition of $n_u$).
If $n_u \leq k_{\mathscr T}$ then clearly $U_{e_u}$ has at most $k_{\mathscr 
T}+1$ number of non-zero coefficients. Suppose now that $n_u>k_\mathscr T$. 
By Lemma \ref{facts}(ii), $S_{\lambda'}^{*k}e_u = \alpha_{\lambda, k} e_{\mathsf{par}^{\langle 
k \rangle}(u)}$ for some scalar $\alpha_{\lambda, k}.$ Let $k$ be an integer 
such that $0\leq k \leq n_u-k_\mathscr T-1$.
Then $n_{\mathsf{par}^{\langle 
k \rangle}(u)} = n_u - k  \geq k_{\mathscr T} + 1,$ and hence 
${\mathsf{par}^{\langle 
k \rangle}(u)} \notin W_0.$
By Lemma 
\ref{branch}(i), 
$$P_ES_{\lambda'}^{*k}e_u=\alpha_{\lambda, k} P_Ee_{\mathsf{par}^{\langle 
k \rangle}(u)}=0.$$
Hence total 
number of possible non-zero coefficients in above expression of $U_{e_u}$ are 
$n_u+1-(n_u-k_\mathscr T)=k_\mathscr T+1$. Thus for each $u\in V$, $U_{e_u}$ 
is a polynomial in $z$ with at most $k_\mathscr T+1$ non-zero coefficients. 
This 
completes the proof of the theorem.
\end{proof}
\begin{remark}
The proof above actually shows that the set of analytic bounded point evaluation of $\mathscr H$ contains the disc $\mathbb D_{r_{\lambda}}$ (refer to \cite[Chapter II, Section 7]{Co}). 
\end{remark}

We conclude this section with a brief discussion on a possible line of investigation. Note that the proof of Theorem \ref{thm1} relies on the notion of Cauchy dual operator, present already in Shimorin's  construction of an analytic model for a left inertible analytic operator \cite{Shimorin}. In case $\dim E=1$, the analytic model can be replaced by the model in which the weighted shift operator can be realized as the operator of multiplication by $z$ on a Hilbert space of formal power series \cite{S} (refer also to \cite{K}). It would be interesting to find a counter-part of the later model in case $\dim E > 1.$  In this regard, the authors would like to draw reader's attention to \cite[Theorems 2.12 and 2.13]{SV} in which it is shown that the adjoint of an arbitrary cyclic operator can be modelled as a backward shift on a reproducing kernel Hilbert space.

\section{Examples}

In this section, we illustrate Theorem \ref{thm1} with the help of several interesting examples. In particular, we see that various directed trees (discrete structures) render to analytic multi-diagonal kernels (analytic structures). These include mainly tridiagonal and  pentadiagonal kernels (the reader is referred to \cite{AM}, \cite{AM-1}, \cite{AFM}, \cite{AMSS}  for a systematic study of scalar-valued and matrix-valued kernels of finite bandwidth; refer
also to \cite{KM} for a new class of matrix-valued kernels on the unit disc arising in the classification problem of homogeneous operators). 
All important examples are summarized in the form of a table at the end of this section.

\begin{example}[(Tridiagonal)] \label{tridiagonal}
Consider the directed tree $\mathscr T_2$ with set of vertices 
$$V:=\{(0,0)\}\cup\{(1,i), (2,i):i\geq1\}$$ and
$\mathsf{root}=(0,0)$. We further require that $\mathsf{Chi}(0,0)=\{(1,1),(2,1)\}$ and 
\[\mathsf{Chi}(1,i)=\{(1,i+1)\},\
\mathsf{Chi}(2,i)=\{(2,i+1)\},\ \text{for all}\ i\geq1. \]
%Then the weighted shift $S_\lambda$ on the directed tree
%$\mathscr T_2$ is given by
%\begin{eqnarray*} 
%S_\lambda(f) &=& f(0,0)(\lambda_{(1,1)}e_{(1,1)} + \lambda_{(2,1)}e_{(2,1)}
%) \\ &+& \displaystyle\sum_ { i\geq1 }
%f(1,i)\lambda_{(1,i+1)}e_{(1,i+1)}+f(2,i)\lambda_{(2,i+1)}e_{(2,i+1)}.\end{eqnarray*}
%Assume that $S_{\lambda}$ is left-invertible. Then $S_{\lambda'}$ exists and
Let $S_{\lambda}$ be a left-invertible weighted shift on $\mathscr T_2.$
It is easy to see from \eqref{kernel} that
$$E:= \ker S_\lambda^* = \{\alpha e_{(0,0)} + \beta(\lambda_{(2, 1)}e_{(1,1)}- 
\lambda_{(1, 1)}e_{(2,1)}) : \alpha, \beta \in \mathbb C\}.$$ 
Also, $V_\prec=\{(0,0)\}$ and $k_{\mathscr T_2}=1$. Therefore, by Theorem
\ref{thm1}, $\kappa_{\mathscr H}(\cdot,\cdot)$ takes the form
\[\kappa_{\mathscr H}(z,w)=I_E+\displaystyle\sum_{\underset{|j-k|\leq
1}{j,k\geq1}}C_{j,k}z^j\overline{w}^k~(z, w \in \mathbb D_{r_{\lambda}}),\]
where $C_{j,k}$ is given by 
$C_{j,k}=P_ES^{*j}_{\lambda'}S_{\lambda'}^k|_E~(j, k =1, 2, \cdots)$.
Let us find an explicit expression for the radius of convergence $r_{\lambda}$ for $S_{\lambda}$.
Note that $$W_n = \{(1, n), (2, n), (1, n+1), (2, n+1)\}$$ for $n \geq 1,$  and 
hence by \eqref{radius},
\beq \label{tri-rad}
r_{\lambda}=\liminf_{n \rar \infty} \left(\sum_{j=1}^2 \left[\big(\lambda'_{(j, n)} \cdots  
\lambda'_{(j, 1)}\big)^2 + \big(\lambda'_{(j, n+1)} \cdots  
\lambda'_{(j, 2)}\big)^2 \right]\right)^{-\frac{1}{2n}},
\eeq
where the sequence 
$\{\lambda'_{(j, n)}\}_{n \geq 1}$ for $j=1, 2$ is given by 
\beq \label{seq-dual}
\lambda'_{(j, n)} = \begin{cases} 
\frac{\lambda_{(j, n)}}{\lambda^2_{(1, n)} + \lambda^2_{(2, n)}}~\mbox{if~}n=1 \\
\frac{1}{\lambda_{(j, n)}}~\mbox{if~}n \geq 2.\end{cases}
\eeq
In this case, the reproducing kernel $\kappa_{\mathscr H}(\cdot,\cdot)$ is
{\it tridiagonal}. Finally, we note that the weight sequence $\lambda$ can be chosen so that $\kappa_{\mathscr H}(\cdot,\cdot)$ is not diagonal. In fact, a routine calculation shows that \beq \label{C21} C_{2, 1}(\lambda_{(2, 1)}e_{(1,1)}- 
\lambda_{(1, 1)}e_{(2,1)}) = \frac{\lambda_{(1, 1)}\lambda_{(2, 1)}}{\lambda^2_{(1, 1)} + \lambda^2_{(2, 1)}}\left(\frac{1}{\lambda^2_{(1, 2)}} - \frac{1}{\lambda^2_{(2, 2)}}\right)e_{(0, 0)},\eeq which is clearly non-zero in case $\lambda_{(1, 2)} \neq \lambda_{(2, 2)}.$ 
\end{example}

The tridiagonal kernel $\kappa_{\mathscr H}$ appearing in Example \ref{tridiagonal} takes a concrete form for a family of weighted shifts $S_{\lambda}$.

\begin{proposition} Let $\mathscr T_2$ and $S_{\lambda}$ be as discussed in Example \ref{tridiagonal}. Let $x:=e_{(0, 0)},  y:=\lambda_{(2, 1)}e_{(1,1)}- 
\lambda_{(1, 1)}e_{(2,1)}.$ Assume that the weight sequence $\{\lambda_{(j, i)} : i \geq 1, j=1, 2\}$ of $S_{\lambda}$ satisfies the following:
\begin{enumerate}
\item[(i)] $\lambda_{(1, 1)} = \lambda_{(2, 1)}$,
\item[(ii)] $\lambda_{(1, 2)} \neq \lambda_{(2, 2)}$,
\item[(iii)] 
$\lambda_{(1, n)} \cdots \lambda_{(1, 2)} = \lambda_{(2, n)} \cdots \lambda_{(2, 2)}$ for every integer $n \geq 3.$
%\item[(iii)] $\lambda_{(1, 1)}\lambda_{(1, 2)}=\lambda_{(2, 1)}\lambda_{(2, 2)}.$
\end{enumerate}
Then the reproducing kernel $\kappa_{\mathscr H}$ takes the form 
\beqn
\kappa_{\mathscr H}(z,w) &=& I_E + \alpha(x \otimes y\, z^2\overline{w} + y \otimes x\, z\overline{w}^2) \\ &+&     \sum_{k=1}^{\infty} \Big({\alpha_k}\, x \otimes x + {\alpha_{k+1}}\, y \otimes y \Big) z^k \overline{w}^k
~(z, w \in \mathbb D_{r_{\lambda}}),
\eeqn
where $\alpha:=\frac{\lambda^2_{(1, 1)}}{\|y\|^4}\left({\lambda^{-2}_{(1, 2)}} - {\lambda^{-2}_{(2, 2)}}\right)$ is a non-zero real number, and 
\beqn \alpha_{k}:= \begin{cases}
{
\|y\|^{-2}}~\mbox{if~}k=1,\\
{\lambda^2_{(1, 1)}}{\|y\|^{-4}}\Big( {\lambda^{-2}_{(1, 2)}} + {\lambda^{-2}_{(2, 2)}} \Big)~\mbox{if~}k=2,\\
{\|y\|^{-2}} \big({\lambda_{(1, k)} \cdots \lambda_{(1, 2)}}\big)^{-2}~\mbox{if~} k \geq 3.\end{cases}
\eeqn 
\end{proposition}
\begin{proof}
As seen in Example \ref{tridiagonal}, $\kappa_{\mathscr H}(\cdot,\cdot)$ is given by
\[\kappa_{\mathscr H}(z,w)=I_E+\displaystyle\sum_{\underset{|j-k|\leq
1}{j,k\geq1}}C_{j,k}z^j\overline{w}^k~(z, w \in \mathbb D_{r_{\lambda}}).\]
Since $C_{2,1}e_{(0, 0)}=0$, by \eqref{C21} and (i), $C_{21}$ is the rank one operator $\alpha\,x \otimes y.$
Note that $\alpha \neq 0$ in view of (ii).
Also, $C_{1, 2} = C^*_{2, 1} = \alpha\,y \otimes x.$

We claim that $C_{k, k+1}=0=C_{k+1, k}$ for all integers $k \geq 2.$ Let us first compute the diagonal operator $S^{*k}_{\lambda'}S^k_{\lambda'}$. Fix an integer $k \geq 2.$ It may be concluded from \cite[Lemma 6.1.1]{Jablonski} that 
\beq \label{power} S^{*k}_{\lambda'}S^k_{\lambda'}e_{u} = \sum_{v \in \childn{k}{u}} \big(\lambda'_v \lambda'_{\parent v}\cdots \lambda'_{\parentn{k-1}{v}} \big)^2e_u.\eeq
Since $S^*_{\lambda'}x=0,$ it follows from \eqref{power} that $C_{k+1, k}x = 0$. Note further that
\beqn
S^{*k}_{\lambda'}S^k_{\lambda'}y &=& \lambda_{(2, 1)}S^{*k}_{\lambda'}S^k_{\lambda'}e_{(1, 1)} - \lambda_{(1, 1)}
S^{*k}_{\lambda'}S^k_{\lambda'}e_{(2, 1)}\\
&=&
\lambda_{(2, 1)}\big(\lambda'_{(1, k+1)} \cdots \lambda'_{(1, 2)}\big)^2 e_{(1, 1)} - \lambda_{(1, 1)}
\big(\lambda'_{(2, k+1)} \cdots \lambda'_{(2, 2)}\big)^2e_{(2, 1)} \\
&\overset{\mbox{(iii)}}=& \frac{1}{\big(\lambda_{(1, k+1)} \cdots \lambda_{(1, 2)}\big)^2}\, y,
\eeqn
where in the last step we used \eqref{seq-dual}. 
This immediately yields that $C_{k+1, k}y = 0$.  This completes the verification of the claim. 
The above calculation also shows that 
$C_{k, k}y = \alpha_{k+1} \|y\|^2 y$ for all integers $k \geq 2.$
%Note further that $C_{1,1}x = \sum_{j=1}^2 \lambda'^2_{(j, 1)} \,x=\frac{1}{\lambda^2_{(1, 1)}+\lambda^2_{(2, 1)}}\,x.$ Thus
%$C_{1, 1}$ is the rank two operator given by
%\beqn
%C_{1, 1} = \frac{1}{\lambda^2_{(1, 1)}+\lambda^2_{(2, 1)}} \left(x \otimes x + \frac{1}{\lambda^2_{(1, 2)}} y \otimes y \right).
%\eeqn 
Since $S^{*}_{\lambda'}S_{\lambda'}y = \frac{\lambda_{(2, 1)}}{\lambda^2_{(1, 2)}} e_{(1, 1)} - \frac{\lambda_{(1, 1)}}{\lambda^2_{(2, 2)}} e_{(2, 1)},$ we have
\beqn
C_{1, 1}y &=& P_{E}\left(\frac{\lambda_{(2, 1)}}{\lambda^2_{(1, 2)}} e_{(1, 1)} - \frac{\lambda_{(1, 1)}}{\lambda^2_{(2, 2)}} e_{(2, 1)}\right)\\ &=& \frac{\lambda_{(2, 1)}}{\lambda^2_{(1, 2)}} \inp{e_{(1, 1)}}{{y}}\frac{y}{\|y\|^2} - \frac{\lambda_{(1, 1)}}{\lambda^2_{(2, 2)}}\inp{e_{(2, 1)}}{y}\frac{y}{\|y\|^2} \\
&=& \left(\frac{\lambda^2_{(2, 1)}}{\lambda^2_{(1, 2)}} + \frac{\lambda^2_{(1, 1)}}{\lambda^2_{(2, 2)}}\right)\frac{y}{\|y\|^2} \overset{\mbox{(i)}}= {\alpha_2}{\|y\|^2} y.
\eeqn
%which is zero in view of (iii).
To compute the diagonal entry $C_{k, k}$,  by \eqref{power}, for any integer $k \geq 1,$
\beqn C_{k,k}x = \sum_{j=1}^2 \big(\lambda'_{(j, k)} \cdots \lambda'_{(j, 1)}\big)^2 x =\alpha_k x,
%= \sum_{j=1}^2 {\alpha_{k}} \frac{\lambda^2_{(j, 1)}}{\|y\|^4}\,x  =  \frac{{\alpha_{k}}}{\|y\|^2}\,x. 
\eeqn
where we used \eqref{seq-dual} and (iii).
Now it is easy to see that the rank two operator $C_{k, k}$ is  given by
$C_{k,k}=
\Big({\alpha_k}\, x \otimes x + {\alpha_{k+1}}\, y \otimes y \Big)
$ for all integers $k \geq 1.$
\end{proof}
\begin{example}
The preceding proposition is applicable to $S_{\lambda}$ with weights 
$\lambda_{(1, 1)}=\lambda_{(2, 1)}= \lambda_{(1, 2)}=1, \lambda_{(2, 2)}=\sqrt{2}=\lambda_{(1, 3)},$ and $\lambda_{(2, 3)}=1=\lambda_{(j, i)}$ for $i \geq 4$ and for $j=1, 2.$
In this case, $\alpha =\frac{1}{8}, \alpha_1 = \frac{1}{2}$, $\alpha_2 = \frac{3}{8}$ and $\alpha_k = \frac{1}{8}$ for all integers $k \geq 3.$
Thus the reproducing kernel $\kappa_{\mathscr H}$ takes the form 
\beqn
\kappa_{\mathscr H}(z,w) &=& I_E + \frac{1}{8}\big(x \otimes y\, z^2\overline{w} + y \otimes x\, z\overline{w}^2 \big) +
\frac{1}{2}\Big(x \otimes x + \frac{3}{4}y \otimes y\Big) z \overline{w} \\ &+& \frac{1}{8}\Big(  3 x \otimes x + y \otimes y \Big) z^2 \overline{w}^2  + \frac{1}{8} \sum_{k=3}^{\infty}  \big(x \otimes x + y \otimes y\big) z^k \overline{w}^k
~(z, w \in \mathbb D_{r_{\lambda}}),
\eeqn  
where, in view of \eqref{tri-rad}, $r_{\lambda}$ can be easily seen to be equal to $1$. 
Also, one may easily deduce from the proof of Theorem \ref{thm1}(iv) that for all integers $i \geq 1$, 
\beqn U_{e_{(j, i)}}(z)=(P_ES_{\lambda'}^{*i-1}e_{(j, i)})z^{i-1} + (P_ES_{\lambda'}^{*i}e_{(j, i)})z^{i},~j=1, 2.\eeqn
It is now easy to see that the orthonormal basis for the reproducing kernel Hilbert space $\mathscr H$ associated with $\kappa_{\mathscr H}$ is given by 
\beqn
\{x, p(z), zp(z)\} \cup \left \{\frac{1}{\sqrt{2}}z^{k-1}p(z)\right \}_{k \geq 3} \cup \{q(z)\} \cup \left \{\frac{1}{\sqrt{2}}z^{k-1}q(z)\right \}_{k \geq 2},
\eeqn
where $p(z) = \frac{1}{2}(y+ x z)$ and $q(z) = \frac{1}{2}(xz -y)$ are linear $E$-valued polynomials.
\end{example}

\begin{example}[(Pentadiagonal)] \label{pentadiagonal}
Consider the directed tree $\mathscr T_3$ with set of vertices $V=\{(0,0),
(1,1)\}\cup\{(2,i), (3,i):i\geq1\}$ and $\mathsf{root}=(0,0)$.
We further require that $\mathsf{Chi}(0,0)=\{(1,1)\}$, $\mathsf{Chi}(1,1)=\{(2,1),(3,1)\}$ and 
\[\mathsf{Chi}(2,i)=\{(2,i+1)\},\
\mathsf{Chi} (3,i)=\{(3,i+1)\},\ \text{for all}\ i\geq1. \]
%Then the weighted shift $S_\lambda$ on the directed tree
%$\mathscr{T}_3$ is given by
%\begin{eqnarray*} 
%S_\lambda(f) &=& f(0,0)\lambda_{(1,1)}e_{(1,1)} +
%f(1,1)(\lambda_{(2,1)}e_{(2,1)}+\lambda_{(3,1)}e_{(3,1)}) \\ 
%&+& \displaystyle\sum_ { i\geq1 }
%f(2,i)\lambda_{(2,i+1)}e_{(2,i+1)}+f(3,i)\lambda_{(3,i+1)}e_{(3,i+1)}.
%\end{eqnarray*}
%Assume that $S_{\lambda}$ is left-invertible. Then $S_{\lambda'}$ exists and
Let $S_{\lambda}$ be a left-invertible weighted shift on $\mathscr T_3.$
As in the preceding example, one can see that
$$E:= \ker S_\lambda^* = \{\alpha e_{(0,0)} + \beta(\lambda_{(3, 1)}e_{(2,1)}- 
\lambda_{(2, 1)}e_{(3,1)}) : \alpha, \beta \in \mathbb C\}.$$
Also $V_\prec=\{(1,1)\}$ and $k_{\mathscr T_3}=2$. Therefore, from Theorem
\ref{thm1}, $\kappa_\mathcal{H}(\cdot,\cdot)$ takes the form
\[\kappa_\mathcal{H}(z,w)=I_E+\displaystyle\sum_{\underset{|j-k|\leq
2}{j,k\geq1}}C_{j,k}z^j\overline{w}^k~(z, w \in \mathbb D_{r_{\lambda}}).\]
Moreover, for $j\geq1$, $S_{\lambda'}^jE \subseteq \mbox{span}~ \Big\{e_v : v
\in \mathsf{Chi}^{\langle j \rangle} \big\{(0,0), (2,1), (3,1) \big\}\Big\}$.
Therefore, 
$$S_{\lambda'}^{*j+1} S_{\lambda'}^j E \subseteq \mbox{span}~
\Big\{e_v : v \in \mathsf{par} \big\{(0,0), (2,1), (3,1)
\big\}\Big\}= \mbox{span}~\{e_{(1,1)}\},$$ 
which gives that
$P_ES^{*j+1}_{\lambda'}S_{\lambda'}^j|_E=0.$ 
%Similarly, one can check that 
%$S_{\lambda'}^{*j-1} S_{\lambda'}^j E \subseteq \mbox{span}~
%\Big\{e_v : v \in \mathsf{Chi} \big\{(0,0), (2,1), (3,1)
%\big\}\Big\}= \mbox{span}~ \Big\{e_v : v \in \big\{(1,1), (2,2),
%(3,2) \big\}\Big\}$, and hence $P_ES^{*j-1}_{\lambda'}S_{\lambda'}^j|_E=0.$
Thus $C_{j,k}=0$ if $|j-k|=1$. Therefore from above,
$\kappa_\mathcal{H}(\cdot,\cdot)$ becomes
\[\kappa_\mathcal{H}(z,w)=I_E+\displaystyle\sum_{\underset{|j-k|=0,
2}{j,k\geq1}}C_{j,k}z^j\overline{w}^k~(z, w \in \mathbb D_{r_{\lambda}}).\]
Since $W_n = \{(2, n-1), (3, n-1), (2, n), (3, n), (2,
n+1), (3, n+1)\}$ for $n 
\geq 2,$ the radius of convergence $r_{\lambda}$ for $S_{\lambda}$ is given by \beqn 
\liminf_{n \rar \infty} \left(\sum_{j=2}^3 \left[\big(\lambda'_{(j, n-1)} \cdots  
\lambda'_{(1, 1)}\big)^2 + \big(\lambda'_{(j, n)} \cdots  
\lambda'_{(j, 1)}\big)^2 + \big(\lambda'_{(j, n+1)} \cdots  
\lambda'_{(j, 2)}\big)^2 \right]\right)^{-\frac{1}{2n}}. 
\eeqn
In this case, the reproducing kernel $\kappa_\mathcal{H}(\cdot,\cdot)$ is
{\it pentadiagonal}.
\end{example}

%One may tempt to conjecture from Example \ref{pentadiagonal} that if
%there exists a positive integer $m$ such that $1 \leq m < k_{\mathscr T}$ and
%$e_v$ is orthogonal to $E$ for all $v \in \mathsf{Chi}^{\langle m
%\rangle}(\mathsf{root})$, then $C_{j,k}=0$ for all $j,k\geq1$ with $|j-k|=m$.
%This is not true in general. 
%
%\begin{example}[Septadiagonal]
%Consider the directed tree
%$\mathscr T$ with set of vertices $$V=\{(0,0), (1,0), (1,1), (1,2)\} \cup 
%\{(i,j):i \geq 2,\, 1 \leq j \leq 4\}$$ and $\mathsf{root}=(0,0)$. We further require that 
%\beqn
%\mathsf{Chi}(0,0)&=& \{(1,0)\},\ \mathsf{Chi}(1,0)=\{(1,1), (1,2)\},\\ 
%\mathsf{Chi}(1,1)&=&\{(2,1), (2, 2)\},\ \mathsf{Chi}(1,2)=\{(2,3), (2,4)\},\eeqn 
%%$\mathsf{Chi}(2,1)=\{(3,1), (3,2)\}$, $\mathsf{Chi}(2,2)=\{(3,3), (3,4)\}$, 
%%$\mathsf{Chi}(2,3)=\{(3,5), (3,6)\}$, $\mathsf{Chi}(2,4)=\{(3,7), (3,8)\}$ 
%and 
%$\mathsf{Chi}(i,j)=\{(i+1,j)\}$ for $i \geq 2,~ 1 \leq j \leq 4$. Then 
%$k_{\mathscr T}=3$ and
%\beqn E &=& \{\alpha e_{(0,0)} + \beta(\lambda_{(1, 
%2)}e_{(1,1)}- 
%\lambda_{(1, 1)}e_{(1,2)}) + \gamma(\lambda_{(2, 
%2)}e_{(2,1)}- 
%\lambda_{(2, 1)}e_{(2,2)}) \\ &+& \delta(\lambda_{(2, 
%4)}e_{(2,3)}- 
%\lambda_{(2, 3)}e_{(2,4)}): \alpha, \beta, \gamma, \delta \in \mathbb C\}. 
%%\\ &\subseteq & \mbox{span}~\Big\{e_v:v \in \{\mathsf{root}\} 
%%\cup \bigcup_{i=2}^3 \mathsf{Chi}^{\langle i\rangle}(\mathsf{root})\Big\}.
%\eeqn
%Arguing along the lines of Example \ref{pentadiagonal}, it is not difficult to verify 
%that $C_{j,k} \neq 0$ for all $j,k \geq 1$ with $|j-k| \leq 3$.
%\end{example}

The invariant $k_{\mathscr T}$ may be bigger than $\dim E$ as shown below.
\begin{example}[(Septadiagonal)] Consider the directed tree $\mathscr T_4$ with set of vertices $V=\{(0,0),
(1, 1), (2, 2)\}
\cup\{(3,i), (4,i):i\geq1\}$ and $\mathsf{root}=(0,0)$. 
We further require that
$\mathsf{Chi}(0,0)=\{(1,1)\}$, $\mathsf{Chi}(1,1)=\{(2,2)\}$, $\mathsf{Chi}(2,2)=\{(3,1), (4, 1)\}$, and 
\[\mathsf{Chi}(3,i)=\{(3,i+1)\},\
\mathsf{Chi} (4,i)=\{(4,i+1)\},\ \text{for all}\ i\geq1. \]
It is easy to see that $\dim\, \ker S^*_{\lambda} = 2, k_{\mathscr T_4}=3$.
The kernel $\kappa_{\mathscr H}$ in this example is septadiagonal. 
We leave the details to the reader. 
\end{example}

The main result also applies to a directed tree which is not locally finite. 
\begin{example} \label{tree-inf}
Consider the directed tree $\mathscr T_\infty$ with set of vertices
$V=\{(i,j): i, j \geq 0\},$ and $\mathsf{root}=(0,0)$. We further require that
\begin{center}
$\mathsf{Chi}(i,j)=\begin{cases}
                   \{(1,k) : k \geq 0\}& \text{if}\ (i,j) = \mathsf{root},\\
                    \{(i+1, j)\}& \text{otherwise}.
                  \end{cases}$

\end{center}
Let $S_\lambda$ be a bounded left-invertible weighted shift on $\mathscr
T_\infty$. Then $E=\ker S_\lambda^*$ is of infinite dimension. Also
$V_\prec=\{\mathsf{root}\}$, and hence $k_{\mathscr T_\infty}=1$.
By Theorem \ref{thm1}, the reproducing kernel 
$\kappa_\mathscr
H$ is tridiagonal.
\end{example}
\begin{remark} \label{direct-sum}
In general, $S_{\lambda}$ is not unitarily equivalent to orthogonal direct sum of unilateral weighted shifts. To see this, consider the weighted shift $S_{\lambda}$ on $\mathscr T_{\infty},$ and
suppose that $S_{\lambda}$ is unitarily equivalent to direct sum $T:=\oplus_{i=1}^{\infty}T_i$ of unilateral weighted shifts $T_i.$ Choose weights of $S_{\lambda}$ such that $\lambda_{(2, 0)} \neq \lambda_{(2, 1)}.$ In this case,
%there exists linearly independent $g_1, g_2 \in \ker S^*_{\lambda}$ such %that 
$$\inp{S^2_{\lambda}g_1}{S_{\lambda}g_2} = \lambda_{(1, 0)}\lambda_{(1, 1)}(\lambda^2_{(2, 0)} - \lambda^2_{(2, 1)}) \neq 0,$$ 
where $g_1 = e_{\mathsf{root}}$ and $g_2 = \lambda_{(1, 1)}e_{(1, 0)} - \lambda_{(1, 0)}e_{(1, 1)}$ belong to $\ker S^*_{\lambda}$.
However, $\inp{T^mX}{T^nY} = 0$ for any $X, Y \in \ker T^*$ and for any positive integers $m, n$ such that $m \neq n.$ 
\end{remark}

%\begin{example}
%Consider the directed tree $\mathscr T$ with set of vertices $$V:=\{(0,0), (1,0),
%(1,2)\} \cup \{(i,j): i \geq 2,\ 0 \leq j \leq 2 \}$$
%and $\mathsf{root}=(0,0)$. We further require that $\child{(0,0)}=\{(1,0),(1,2)\},
%\child{(i,2)}=\{(i+1,2)\}\ \mbox{for}\ i \geq 1,
%\child{(1,0)}=\{(2,0),(2,1)\}, \child{(i,j)}=\{(i+1,j)\}$ for $i \geq 2$ and
%$j=0,1$. Let $S_{\lambda}$ be a weighted shift on directed tree.
%In the following discussion, we need the following formula for the kernel of positive integral powers of $S^*_{\lambda}$, which can be derived as in the proof of \cite[Proposition 3.5.1]{Jablonski}:
% 
%\end{example}

%We summarize the above examples in the following table.

{\small
\begin{table}[H]
\caption{}
% title of Table
\begin{center}
% used for centering table
\begin{tabular}{| c | c | c | c |}
% centered columns (4 columns)
\hline
%inserts double horizontal lines
Directed Tree  $\mathscr T$ & Dimension of $\ker S^*_{\lambda}$  &
$k_{\mathscr T}$  & Form of $\kappa_{\mathscr H}(z, w)$ \\ \hline
% inserts table
%heading
% inserts single horizontal line
$\mathscr T_1$ & $1$ & $0$ & \mbox{diagonal} \\ \hline 
$\mathscr T_2$ & $2$ & $1$ & \mbox{tridiagonal} \\ \hline
$\mathscr T_{3}$ & $2$ & $2$ & \mbox{pentadiagonal} \\ \hline
$\mathscr T_{4}$ & $2$ & $3$ & \mbox{septadiagonal} \\ \hline
$\mathscr T_{\infty}$ & $\infty$ & $1$ & \mbox{tridiagonal}\\ \hline

%inserts single line
\end{tabular}
\end{center}

%\label{table:nonlin}
% is used to refer this table in the text
\end{table}
}

\section{Spectral Picture of $S_{\lambda}$}

In this section, we use analytic model constructed in Sections 2 and 3 to discuss spectral theory of weighted shifts $S_{\lambda}$ on rooted directed trees. This part has an overlap with \cite[Theorems 2.1 and 2.3]{C}, 
where the spectral picture of certain weighted composition operators is described. However,
the conclusion of (i)-(iii) of Theorem \ref{spectral} can not be deduced from the aforementioned results of \cite{C} as the directed trees considered in this part need not be locally finite. On the other hand, in the context of rooted directed trees, weighted shifts always have connected spectrum. This is in contrast with \cite[Example 5]{C}, where a composition operator with disconnected spectrum has been constructed.
Positively, the power of analytic model comes into the picture while computing the point spectra of $S_{\lambda}$ and $S^*_{\lambda}$. In this regard, the rather technical proof of \cite[Theorem 2.1]{C} 
should be compared with that of (i) and (ii) of Theorem \ref{spectral}. 
%Also, unlike the treatment of \cite{C}, we do not confine ourselves to locally finite directed trees.

Before we state the main result of this section, we recall a couple of known facts about $S_{\lambda}$. 

Any weighted shift $S_{\lambda}$ on a directed tree is {\it circular} \cite[Theorem 3.3.1]{Jablonski}: 
{\it For every $\theta \in \mathbb R$, there exists a unitary $U_\theta$ on $l^2(V)$ such that $U_\theta S_{\lambda} = e^{i \theta} S_{\lambda} U_\theta$.} An immediate consequence of this shows that all spectral parts of $S_{\lambda}$ have circular symmetry about $0$ \cite[Corollary 3.3.2]{Jablonski}.

%In the proof of Theroem \ref{spectral}, we use different realizations of the weighted shift $S_{\lambda}$ interchangebly.

%In the following result, we use the analytic model of $S_{\lambda}$ to describe various spectral parts of left-invertible $S_{\lambda}$ including
%the spectrum (cf. \cite[Section 5]{S}).

Here is the statement of the main result of this section.
\begin{theorem} \label{spectral}
Let $S_\lambda \in B(l^2(V))$ be a left-invertible weighted shift on $\mathscr T$ and let $E:={\ker} S_\lambda^*$.
Then we have the following.
\begin{enumerate}
\item[(i)] The point spectrum $\sigma_p(S_\lambda)$ of $S_{\lambda}$ is empty.
\item[(ii)] If $r_\lambda$ is the radius of convergence for $S_{\lambda}$ then $$\mathbb D_{r_\lambda} \subseteq \sigma_p(S^*_{\lambda}) \subseteq \sigma(S_{\lambda}) = 
 \overline{\mathbb D}_{r(S_{\lambda})}.$$ 
 \item[(iii)] 
 $\bigvee\{\ker (S_\lambda^*- w): w \in \mathbb D_\epsilon\}=l^2(V)$ for every positive number $\epsilon$.
\end{enumerate}
If, in addition, $E$ is finite dimensional then
\begin{enumerate}
\item[(iv)] $\sigma_{ap}(S_{\lambda})=\sigma_e(S_{\lambda})$ is a union of at most $\dim E$ number of annuli centered at the origin.
\item[(v)] 
the Fredholm index $\mbox{ind}\, (S_{\lambda}-w)$ of $S_{\lambda}-w$ is at least $-\dim E$ on any connected component of $\mathbb C \setminus \sigma_e(S_{\lambda})$.
Moreover, $\mbox{ind}\, (S_{\lambda}-w)$ is exactly $-\dim E$ on the connected component of $\mathbb C \setminus \sigma_e(S_{\lambda})$ that contains $0$. 
\item[(vi)] for any positive integer $k,$ $$\dim \big(\ker S^{*k}_\lambda/\ker S^{*k-1}_\lambda \big)=\dim E.$$
\end{enumerate}
\end{theorem}
\begin{remark} 
%The following remarks are in order:
%\begin{enumerate} 
%\item
Since 
$r_\lambda r(S_{\lambda'})\geq1$ (Theorem \ref{thm1}),
by the inclusion in (ii), $r(S_\lambda)r(S_{\lambda'})\geq1$.
This inequality is sharp. In fact, if $S_{\lambda}$ is an isometry then $r(S_\lambda)=1=r(S_{\lambda'}),$ so that equality holds in $r(S_\lambda)r(S_{\lambda'})\geq1$. Also, if $r(S_\lambda)=1=r(S_{\lambda'})$ then $r_{\lambda}$ is necessarily equal to $1.$ Finally, since $S_{\lambda}$ is analytic, the part (vi) above precisely says that $S^*_{\lambda}$ is an abstract backward shift in the sense of \cite{B-1} and \cite{R-1}.  
%It may happen that $r(S_\lambda)r(S_{\lambda'}) = 1$ but $S_{\lambda}$ is non-isometric.
%\item In case $\dim E > 1$, it is not clear whether $\sigma_p(S^*_{\lambda}) \subseteq \overline{\mathbb D}_{r_\lambda}$.   
%%We do not know the exact description
%%the point spectrum of $S^*_{\lambda}$.
%%However, it is likely that $\sigma_p(S^*_{\lambda})$ is bigger than the disc $\mathbb D_{r_{\lambda}}$ (see \cite[Theorem 1 and Proof of Theorem 7]{AM}, where actually observed that the domain of scalar valued kernels of finite bandwidth, the point spectrum of the adjoint of the multiplication operators contains 
%\end{enumerate}
\end{remark}

In the proof of Theorem \ref{spectral}, we need the analytic model as well as a number of general facts about $S_{\lambda}$. 
The first of which generalizes a well-known fact that the spectrum of a weighted shift is connected \cite[Theorem 4]{S}(see also \cite[Theorem 8]{G}, \cite[Theorem 3.5]{R-1}).
\begin{lemma} \label{connected}
The spectrum of an analytic operator is connected.
\end{lemma}
\begin{proof}
Let $T \in B(\mathcal H)$ be analytic. We adapt the technique 
of \cite[Lemma 3.8]{Ch-Ya} to the present situation. Since 
$T$ is analytic, 
\begin{equation}\label{eq1}
\displaystyle\bigvee_{k\geq0}{\ker}\, T^{*k}=\mathcal H.
\end{equation}
Therefore, $0 \in \sigma(T).$
Let $K_1$ be the connected component of $\sigma(T^*)$ containing 0 and 
$K_2=\sigma(T^*)\setminus K_1$. 
If possible, suppose that $K_2$ is non-empty. Then by Riesz 
Decomposition Theorem \cite[Chapter VII, Proposition 4.11]{Conway}, there are closed subspaces $\mathcal H_1$ and 
$\mathcal H_2$ invariant under $T^*$ such that 
$\mathcal H=\mathcal H_1 \oplus \mathcal H_2$, 
$\sigma({T^*}|_{\mathcal H_1})=K_1$ and $\sigma({T^*}|_{\mathcal H_2})=K_2$. 
Let 
$h\in {\ker} {T^*}^{k}$. Then 
$h=x+y$ for $x\in \mathcal H_1$ and $y\in \mathcal H_2$. Since $T^{*k}h=0$, it follows 
that $T^{*k}x=0=T^{*k}y$. If $y$ is 
non-zero, then 
$0\in \sigma_p({T^{*k}}|_{\mathcal H_2})
\subseteq\sigma({T^{*k}}|_{\mathcal H_2}),
$ and hence 
by spectral mapping property,
$0\in\sigma({T^*}|_{\mathcal H_2})=K_2$, which is a contradiction. 
So $y$ must 
be 
zero. Therefore, $\mathcal H_1$ contains $\ker 
T^{*k}$ for all $k \geq 0$. Hence from \eqref{eq1}, we get 
$\mathcal H_1=\mathcal H$, and hence $K_2$ must be empty. This is contrary to the assumption that $K_2 \neq \emptyset$. This shows that $\sigma(T^*)$ is connected. Since $\sigma(T) = \{\overline{z} : z \in \sigma(T^*)\}$ and $z \rightsquigarrow \overline{z}$ is continuous, $\sigma(T)$ is connected.
\end{proof}

\begin{lemma}
\label{decomposition}
Let $S_{\lambda}$ be a weighted shift on $\mathscr T$ and let $d:=\mbox{card}(\childn{k_{\mathscr T}}{\mathsf{root}})$ (possibly infinite). 
%and let $E:=\ker S^*_{\lambda}.$ 
Then there exist subspaces $\mathcal M$ and $\mathcal H_i~(i=1, \cdots, d)$ such that \beq \label{deco}
S_{\lambda}=\left[\begin{array}{ccccc}
A & 0 & 0 & \cdots & 0      \\
A_1 & S_1 & 0 & \cdots & 0     \\
A_2 & 0 & S_2 & \cdots & 0     \\
\vdots & \vdots & \vdots & \ddots & \vdots      \\
A_d & 0 & \cdots & 0 & S_d
\end{array}\right] ~\mbox{on~}l^2(V) = \mathcal M \oplus  \mathcal H_1 \oplus \cdots \oplus \mathcal H_d,
\eeq
where $A:=P_{\mathcal M}S_{\lambda}|_{\mathcal M},$ $A_i:=P_{\mathcal H_i}S_{\lambda}|_{\mathcal M},$ and $S_i:=S_{\lambda}|_{\mathcal H_i}$ for $i=1, \cdots, d.$ Moreover, the following statements hold.
\begin{enumerate}
\item[(i)] Each $S_i$ is unitarily equivalent to a unilateral weighted shift.
\item[(ii)] 
Let $\mathscr T$ possess the property that
$v \in \childn{k_{\mathscr T}-1}{\mathsf{root}}$ whenever  $\mbox{card}(\child v)$ is infinite for some $v \in W_0.$ Then
$S_{\lambda}$ is a finite rank perturbation of $S_1 \oplus \cdots \oplus S_d.$
\end{enumerate}
\end{lemma}
\begin{proof}
Note that for all $v \in \childn{k_{\mathscr T}}{\mathsf{root}}$, $\mbox{card}(\child v)=1.$ Let $W_{-1}$ be as defined in \eqref{Wn}.
We relabel the set $V$ of vertices as follows:
\beqn \label{label}
V= W_{-1} \sqcup \{v_{i, n} : n \geq 0,\ i=1, \cdots, d\} 
\eeqn
such that $\childn{k_{\mathscr T}}{\mathsf{root}}=\{v_{i, 0} : i=1, \cdots, d\}$, and
$\child {v_{i, n}} = \{v_{i, n+1}\}$ for all $n \geq 0,$  $i=1, \cdots, d.$
Now consider the subspaces $\mathcal M$ and $\mathcal H_i$ of $l^2(V)$ given by
\beqn
\mathcal M:=\bigvee\,\{e_v : v \in W_{-1}\},\ \mathcal H_i:=\bigvee \{e_{v_{i, n}} : n \geq 0\},\ i=1, \cdots, d.
\eeqn 
Note that 
%$\mathcal M$ is finite dimensional, and 
the subspaces $\mathcal H_1, \cdots, \mathcal H_d$ are invariant under $S_{\lambda}.$
Then $S_{\lambda}$ admits the decomposition as given by \eqref{deco}.
%\beqn
%S_{\lambda}=\left[\begin{array}{ccccc}
%A & 0 & 0 & \cdots & 0      \\
%A_1 & S_1 & 0 & \cdots & 0     \\
%A_2 & 0 & S_2 & \cdots & 0     \\
%\vdots & \vdots & \vdots & \ddots & \vdots      \\
%A_d & 0 & \cdots & 0 & S_d
%\end{array}\right] ~\mbox{on~}l^2(V) = \mathcal M \oplus \mathcal H_1 \cdots \mathcal H_d,
%\eeqn
%where $A:=P_{\mathcal M}S_{\lambda}|_{\mathcal M},$ $A_i:=P_{\mathcal H_i}S_{\lambda}|_{\mathcal M},$ and $S_i:=S_{\lambda}|_{\mathcal H_i}$ for $i=1, \cdots, d.$ 
Since $S_ie_{v_{i, n}} = \lambda_{v_{i, n+1}} e_{v_{i, n+1}}~(n \geq 0),$ it is clear that each $S_i$ is unitarily equivalent to a unilateral weighted shift. 

To see (ii), note that if $\mbox{card}(W_{-1})$ is infinite then for some $v \in W_{-1} \subseteq W_0$,
we must have $\child v \subseteq W_{-1}$ and $\mbox{card}(\child v)$ is infinite. But then by hypothesis, $v \in \childn{k_{\mathscr T}-1}{\mathsf{root}}$, which implies that $\child v \cap W_{-1} = \emptyset.$ Thus we arrive at a contradiction. This shows that $\mbox{card}(W_{-1})$ is finite, and hence $\mathcal M$ is finite dimensional. Thus $A, A_1, \cdots, A_d$ are finite rank operators, and the conclusion in (ii) is now immediate.  
\end{proof}

The following is certainly known. We include it for the sake of completeness. 
\begin{lemma} \label{f-cyclic} Let $T \in B(\mathcal H)$ be finitely cyclic. 
If $\sigma_p(T)$ is empty then $\sigma_{ap}(T)=\sigma_e(T)$.
\end{lemma}
\begin{proof}
By \cite[Proposition 1(i)]{H}, $\dim \ker(T^*-w)$ is finite for every $w \in \mathbb C.$ If $\sigma_p(T)=\emptyset$ then it is easy to see that  $\sigma_{ap}(T)=\sigma_e(T)$.
\end{proof}

We also need exact description of the kernel of positive integral powers of $S^*_{\lambda}$ in the proof of Theorem \ref{spectral}.
\begin{lemma}\label{lemABS}
Let $S_{\lambda} \in B(l^2(V))$ be a weighted shift on $\mathscr T=(V, \mathcal E)$. Then, for all integers $k \geq 1$,
\beq \label{kernelk}
\ker S^{*k}_\lambda = \bigvee\Big \{e_v : v \in \cup_{i=0}^{k-1} \childn{i}{\mathsf{root}}\Big \} \oplus \bigoplus_{v \in W_{-1}}\Big(l^2 \big(\childn{k}{v} \big) \ominus \langle \lambdab^v_k \rangle\Big),
\eeq
where $\lambdab^v_k : \childn{k}{v} \rar \mathbb C~\mbox{is defined by~} 
\lambdab^v_k(u)= \lambda_u \lambda_{\parent u} \cdots \lambda_{\parentn{k-1}{u}},$ and $W_{-1}$ is given by \eqref{Wn}. Consequently,
\beq\label{dim-kernelk}
\dim \ker S^{*k}_\lambda = \displaystyle \sum_{i=0}^{k-1}\mbox{card}\big(\childn{i}{\mathsf{root}}\big) + \sum_{v \in W_{-1}}\Big(\mbox{card}\big(\childn{k}{v}\big)-1\Big).
\eeq
\end{lemma}
\begin{proof}
Following the lines of the proof of \cite[Proposition 3.5.1]{Jablonski}, one can easily deduce that for all integers $k \geq 1$,
$$\ker S^{*k}_\lambda = \bigvee\Big \{e_v : v \in \cup_{i=0}^{k-1} \childn{i}{\mathsf{root}}\Big \} \oplus \bigoplus_{v \in V}\Big(l^2 \big(\childn{k}{v} \big) \ominus \langle \lambdab^v_k \rangle\Big).$$
From Lemma \ref{branch}(ii), we know that for a vertex $v \in V$, $\mbox{card}(\child v)=1$ if $n_v \geq k_\mathscr T$. This is equivalent to the fact that $\mbox{card}(\childn{m}{v})=1$ for all $m \geq 1$ if $v \notin W_{-1}$. Hence, $l^2 \big(\childn{k}{v} \big) \ominus \langle \lambdab^v_k \rangle=\{0\}$ if $v \notin W_{-1}$. This proves \eqref{kernelk}. The proof of \eqref{dim-kernelk} is obvious in view of \eqref{kernelk}.
\end{proof}

\begin{proof}[Proof of Theorem \ref{spectral}] 
In 
view of Theorem \ref{thm1}, it is sufficient to work with the analytic model $({\mathscr 
M}_z, 
\kappa_{\mathscr H},\mathscr H)$  of $S_{\lambda},$ where the reproducing kernel 
Hilbert space $\mathscr H$ consists of
$E$-valued holomorphic functions $U_f$ on the disc 
$\mathbb D_{r_{\lambda}}$ given by \beqn 
U_f(z):=\displaystyle \sum_{n\geq0}(P_ES_{\lambda'}^{*n}f)z^n. \eeqn 
We check that $\sigma_p(\mathscr M_z)=\emptyset.$ 
Let $w \in \mathbb C$ and $h = \sum_{k=0}^{\infty} a_n z^n \in \mathscr H$ be such that
$(\mathscr M_z - w) h = 0,$ where $\{a_n\}_{n \geq 0} \subseteq E.$
Then for any $g \in E,$ 
\beqn
(z-w)\sum_{k=0}^{\infty} \inp{a_n}{g} z^n = 0~\mbox{for all~}z \in \mathbb D_{r_{\lambda}}.
\eeqn
It follows that $\inp{a_n}{g}=0$ for $g \in E,$ and hence $a_n=0$ every $n \geq 0$. This shows that $h=0,$ which gives (i).

To see (ii), 
%The first inclusion is clear from the proof of Proposition \ref{CD}(ii). 
note that for $f\in l^2(V)$, $g\in E$ and $w \in\mathbb{D}_{r_{\lambda}}$, by Theorem 
\ref{thm1}(i),
\[\langle U_f, {\mathscr M}_z^*\kappa_{\mathscr H}(\cdot,w)g\rangle=\langle 
{\mathscr M}_zU_f, \kappa_{\mathscr H}(\cdot, w)g\rangle=\langle w U_f(w), 
g\rangle
=\langle U_f, \overline{w}\kappa_{\mathscr H}(\cdot,w)g\rangle.\]
Thus ${\mathscr M}_z^*\kappa_{\mathscr H}(\cdot, w)g=\overline{w}k_\mathscr 
H(\cdot,w)g$ for 
all $w \in\mathbb{D}_{r_{\lambda}}$ and $g\in E$. Hence the point spectrum of ${\mathscr 
M}_z^*$ contains $\mathbb D_{r_{\lambda}}$.
As recorded earlier, $S_{\lambda}$ is circular, so that 
$\sigma(S^*_{\lambda})=\sigma(S_{\lambda}).$
The second inclusion in (ii) now follows from $\sigma_p(S^*_{\lambda}) \subseteq \sigma(S^*_{\lambda}).$ To complete the proof of (ii), it is only left to check 
that $\sigma(S_{\lambda}) = 
\overline{\mathbb D}_{r(S_{\lambda})}$.  By Lemmas \ref{lem1} and \ref{connected}, $\sigma(S^*_{\lambda})$ is 
connected. 
Now, suppose that 
$\sigma(S_\lambda)\neq\overline{\mathbb{D}}_{r(S_\lambda)}$. Since 
$r(S_{\lambda}) \in \sigma(S_{\lambda}),$ 
there is a $w_0 \in {\mathbb{D}}_{r(S_\lambda)}$ such that 
$w_0 \in \rho(S_\lambda):=\mathbb C \setminus \sigma(S_{\lambda}).$ Since $\rho(S_\lambda)$ 
is open, there is 
an $\epsilon > 0$ such that 
$\mathbb{D}_\epsilon(w_0):=\{w \in\mathbb{C}:|w-w_0|< \epsilon\}
\subseteq\rho(S_\lambda) \cap {\mathbb{D}}_{r(S_\lambda)}$.  Since 
$S_{\lambda}$ is circular and $\sigma(S_{\lambda})$ is connected, we arrive at 
a contradiction.
Thus, the 
spectrum of $S_\lambda$ is a disk
of radius $r(S_\lambda)$ centred at the origin. 
%The inequality $r(S_\lambda)r(S_{\lambda'})\geq1$ now follows from the fact that $r_{\lambda} \geq r(S_{\lambda'})^{-1}$ as recorded in Lemma \ref{r-lambda}.

%Before we see (iii), recall that any element in $\mathscr H$ is of the form $U_f$ (see \eqref{Uf}).
Suppose that 
$U_f$ is orthogonal to $\bigvee\{\ker ({\mathscr 
M}_z^*-\overline{w}):w\in\mathbb{D}_\epsilon\}$. By the preceding paragraph, 
$\kappa_{\mathscr H}(\cdot,w)g$ belongs to $\ker ({\mathscr 
M}_z^*-\overline{w})$ for every $w \in \mathbb{D}_{r_{\lambda}}.$ Hence 
$$
\sum_{n\geq0}\langle 
P_ES_{\lambda'}^{*n}f,g\rangle w^n = \langle 
U_f(w),g\rangle = \langle 
U_f,\kappa_{\mathscr H}(\cdot,w)g\rangle=0$$ 
$\mbox{for all~} w \in \mathbb 
D_\epsilon \cap \mathbb D_{r_{\lambda}}~\mbox{and~} 
g\in E.$
This implies that $\langle P_ES_{\lambda'}^{*n}f,g\rangle=0$ for all 
$n\geq0$ and $g\in E$. In particular, $\langle 
P_ES_{\lambda'}^{*n}f,P_ES_{\lambda'}^{*n}f\rangle=0$ for all $n\geq0$. Thus 
$U_f=0$. Therefore, 
$\bigvee\{{\ker}({\mathscr M}_z^*-\overline{w}):w\in\mathbb 
D_\epsilon\}=\mathscr 
H$.

Assume now that $E$ is finite dimensional. By Corollary \ref{cyclic}, $S_{\lambda}$ is finitely cyclic. 
By (i) above and Lemma \ref{f-cyclic}, $\sigma_{ap}(S_{\lambda}) = \sigma_e(S_{\lambda}).$
Thus to see (iv), it suffices to check that $\sigma_{e}(S_{\lambda})$ is a finite union of annuli centered at the origin. 
Let $d:=\mbox{card}(\childn{k_{\mathscr T}}{\mathsf{root}}),$ which is finite in view of Proposition \ref{probranch}.
Also, since $E$ is finite dimensional, by the same proposition $\mathscr T$ is locally finite.
Hence by Lemma \ref{decomposition}(ii), there exist unilateral weighted shifts $S_1, \cdots, S_d$ such that
$S_{\lambda}$ is a finite rank perturbation of $S_1 \oplus \cdots \oplus S_d.$
In particular, the essential spectrum of $S_{\lambda}$ equals the union of
essential spectrum of $S_1, \cdots, S_d$ \cite{Conway}. Another application of Lemma \ref{f-cyclic} shows that $\sigma_{ap}(S_i)=\sigma_e(S_i)$. However, the approximate point spectrum of a unilateral weighted shift is necessarily an annulus centered at the origin \cite[Theorem 1]{R}. The desired conclusion in (iv) is now immediate. 

Let us now see part (v). For any $w \in \mathbb C,$ note that \beqn \mbox{ind}\, (S_{\lambda} - w) = \mbox{ind}\, \oplus_{i=1}^d (S_{i}-w) 
%= -\dim \ker (\oplus_{i=1}^d (S^*_{i} - \bar{w}) 
\overset{\mbox{(i)}}= -  \oplus_{i=1}^d \dim \ker (S^*_{i}-\overline{w}). \eeqn
Since $\dim \ker (S^*_{i}-\overline{w})$ at most one, $\mbox{ind}\, (S_{\lambda} - w)$ is at least $-\dim E.$ However, 
$\dim \ker S^*_{i}=1$ for all $i$, and hence $-\dim E = \mbox{ind}\, S_{\lambda}= -d.$ Note that the proof above shows that $\mbox{card}(\childn{k_{\mathscr T}}{\mathsf{root}}) = \dim E$, and hence by Lemma \ref{branch}(ii), 
\beq \label{chi-ker} \mbox{card}(\childn{k}{\mathsf{root}}) = \dim E~\mbox{for all integers~}k \geq k_{\mathscr T}.\eeq 

To see (vi), fix an integer $k \geq 1$ and let
$\mathscr Q_k:=\ker S^{*k}_\lambda/\ker S^{*k-1}_\lambda$. 
Then using \eqref{dim-kernelk}, we get
\beq \label{quotnt}
\dim \mathscr Q_k 
%&=& \dim \ker S^{*k}_\lambda - \dim \ker S^{*k-1}_\lambda \nonumber\\
%&=& \displaystyle \sum_{i=0}^{k-1}\mbox{card}\big(\childn{i}{\mathsf{root}}\big) + \sum_{v \in W_{-1}}\Big(\mbox{card}\big(\childn{k}{v}\big)-1\Big)\nonumber\\ &-&\displaystyle \sum_{i=0}^{k-2}\mbox{card}\big(\childn{i}{\mathsf{root}}\big) - \sum_{v \in W_{-1}}\Big(\mbox{card}\big(\childn{k-1}{v}\big)-1\Big)\nonumber\\
&=& \mbox{card}\big(\childn{k-1}{\mathsf{root}}\big)+ \sum_{v \in W_{-1}}\mbox{card}\big(\childn{k}{v}\big)
-\sum_{v \in W_{-1}}\mbox{card}\big(\childn{k-1}{v}\big)\nonumber\\
&=&\sum_{v \in W_{-1}}\mbox{card}\big(\childn{k}{v}\big)-\sum_{v \in W_{-1}\setminus \{\mathsf{root}\}}\mbox{card}\big(\childn{k-1}{v}\big)
\eeq
Since $\childn{l}{\mathsf{root}} = \childn{l-1}{\child{\mathsf{root}}}$, it follows that $$\mbox{card}\big(\childn{l}{\mathsf{root}}\big) = \displaystyle \sum_{v \in \child{\mathsf{root}}} \mbox{card}\big(\childn{l-1}{v}\big)$$ for any positive integer $l.$
Therefore, $\sum_{v \in W_{-1}}\mbox{card}\big(\childn{k}{v}\big)$ is equal to
\beqn
 \mbox{card} \big(\childn{k}{\mathsf{root}}\big) + \displaystyle \sum_{v \in \child{\mathsf{root}}} \mbox{card}\big(\childn{k}{v}\big)+ \cdots +  \displaystyle \sum_{v \in \childn{k_\mathscr T-1}{\mathsf{root}}} \mbox{card}\big(\childn{k}{v}\big)\\
= \mbox{card}\big(\childn{k}{\mathsf{root}}\big) + \mbox{card}\big(\childn{k+1}{\mathsf{root}}\big)+ \cdots +  \mbox{card}\big(\childn{k+k_\mathscr T-1}{\mathsf{root}}\big).
\eeqn
Similarly, $\sum_{v \in W_{-1}\setminus \{\mathsf{root}\}}\mbox{card}\big(\childn{k-1}{v}\big)$ is equal to
\beqn
 \mbox{card}\big(\childn{k}{\mathsf{root}}\big)+\mbox{card}\big(\childn{k+1}{\mathsf{root}}\big)+ \cdots +  \mbox{card}\big(\childn{k-1+k_\mathscr T-1}{\mathsf{root}}\big).
\eeqn
Substituting last two identities in \eqref{quotnt}, we get
$$\dim \mathscr Q_k= \mbox{card}\big(\childn{k+k_\mathscr T-1}{\mathsf{root}}\big) \overset{\eqref{chi-ker}}= \dim E.$$
This completes the proof of the theorem.
\end{proof}
\begin{remark}
The identity \eqref{chi-ker}, as established in the proof of Theorem \ref{spectral}(v), comes surprisingly as a consequence of index theory.
As evident, this identity is otherwise difficult to disclose. Note that the left hand side of \eqref{chi-ker} is a variant dependent on $\mathscr T$ while the right hand side of \eqref{chi-ker} depends solely on $S_{\lambda}$.
Further, since $\dim E$ is finite, by Proposition \ref{probranch}, $\mathscr T$ is locally finite and $\mbox{card}(V_\prec) < \infty.$ Therefore, using \eqref{kernel} and \eqref{chi-ker}, one gets the following.
\beqn
\mbox{card}(\childn{k_{\mathscr T}}{\mathsf{root}}) = 1- \mbox{card}(V_\prec)+\sum_{v \in V_\prec}\mbox{card}(\child v).
\eeqn
\end{remark}

One particular consequence of Theorem \ref{spectral}(iv) is that $\sigma_{ap}(S_{\lambda})$ (resp. $\sigma_{e}(S_{\lambda})$) of a weighted shift $S_{\lambda}$ on a directed tree could be {\it disconnected}. 
For instance, in case $\dim E=2,$ by choosing the weight sequence $\lambda$ appropriately (so that the approximate point spectra of $S_1$ and $S_2$, as appearing in the proof of Theorem \ref{spectral}, are disjoint annuli), we can have two connected components of $\sigma_{ap}(S_{\lambda})$. Moreover, the index of $S_{\lambda}- w$ may vary from $-2$ to $0$ on different components of $\mathbb C \setminus  \sigma_e(S_{\lambda})$. Again, in the above situation, 
\beqn
\mbox{ind}\, (S_{\lambda} - w) =\begin{cases} -2~\mbox{on a bounded component of}~\mathbb C \setminus  \sigma_e(S_{\lambda})~\mbox{containing}~0  \\ 
-1~\mbox{on a bounded component of}~\mathbb C \setminus  \sigma_e(S_{\lambda})~\mbox{not containing}~0.
\end{cases}
\eeqn  
This is not possible in case $\dim E =1$ in view of \cite[Theorem 1]{R}.

The conclusion of Theorem \ref{spectral}(iv) need not be true in case $\dim E$ is infinite.

\begin{example} Let $\mathscr T_{\infty}$ be the directed tree as discussed in Example \ref{tree-inf} and let $S_{\lambda}$ be a left-invertible weighted shift on $\mathscr T_{\infty}.$
For a given $\mu > 0,$ choose the weight sequence of $S_{\lambda}$ such that for each $j \geq 0,$ the sequence $\{\lambda_{(i+1, j)}\}_{i \geq 0}$ converges to $\mu.$ As seen in the proof of Lemma \ref{decomposition}, $S_{\lambda}$ is a rank one perturbation of the direct sum of unilateral weighted shifts $S_j$ on $\mathcal H_j$. Thus $\sigma_e(S_{\lambda}) = \sigma_e(\oplus_{j=1}^{\infty}S_j).$ 
Note that $\sigma_e(S_j)=\sigma_{ap}(S_j)$ is the circle of radius $\mu$ centered at the origin \cite{S}. Since
$\sigma_p(S^*_j)=\mathbb D_{\mu}$,
the the essential spectrum of $\oplus_{j=1}^{\infty}S_j$ contains $\mathbb D_{\mu}$. As essential spectrum is always closed, $\overline{\mathbb D}_\mu \subseteq \sigma_e(\oplus_{j=1}^{\infty}S_j).$ 
Also, $$\sigma_e(\oplus_{j=1}^{\infty}S_j) \subseteq \sigma(\oplus_{j=1}^{\infty}S_j) = \overline{\mathbb D}_\mu.$$
This shows that $\sigma_e(S_{\lambda})=\overline{\mathbb D}_\mu.$ 
On the other hand, $0 \notin \sigma_{ap}(S_{\lambda})$ since $S_{\lambda}$ is left-invertible. In particular,
$\sigma_{ap}(S_{\lambda}) \neq \sigma_{e}(S_{\lambda}).$
\end{example}

We see below that $S_{\lambda}$ belongs to the Cowen-Douglas class (refer to \cite{C-D}; refer also to \cite{C-S} for the extended definition of $B_n(\Omega)$ in case $n$ is not finite).
\begin{corollary}\label{CD}
Let $S_\lambda \in B(l^2(V))$ be a left-invertible weighted shift on $\mathscr T$ and let $S_{\lambda'}$ denote the Cauchy dual of 
$S_{\lambda}$. Let $E:=\ker S_\lambda^*$ and $\delta :=\frac{1}{\|S_{\lambda'}\|}$. Then $S_\lambda^*$ belongs to Cowen-Douglas class 
$B_{\dim E}(\mathbb{D}_\delta)$.
%following statements hold:
%\begin{itemize}
% \item[(i)] $\mbox{ran}(S_\lambda^*- \mu)=l^2(V)$ for all $\mu \in \mathbb D_\delta$, where $\delta=\frac{1}{\|S_{\lambda'}\|}$.
%  \item[(ii)] The point-spectrum of $S^*_{\lambda}$ contains $\mathbb D_{r_{\lambda}}$ and $\dim \ker (S_\lambda^*-\mu)=\dim E$ for all $\mu \in \mathbb D_{r_{\lambda}}$, where $r_{\lambda}$ is given by \eqref{radius}.
% \item[(iii)] $\bigvee\{\ker (S_\lambda^*-\mu):\mu \in \mathbb D_\epsilon\}=l^2(V)$, for each positive $\epsilon$.
%\item[(i)] $S_\lambda^*$ belongs to Cowen-Douglas class 
%$B_{\dim E}(\mathbb{D}_\delta)$, 
%where $\delta=\frac{1}{\|S_{\lambda'}\|}$.
% \item[(ii)] $S^*_{\lambda}$ has a dense set of cyclic vectors.
%\end{itemize} 
\end{corollary}

\begin{proof}
Since $S_\lambda$ is left-invertible, 
$$(S_\lambda^*S_\lambda)^{-1}
=S_{\lambda'}^*S_{\lambda'}\leq\|S_{\lambda'}^*S_{\lambda'}\|I.$$
That is, $S_\lambda^*S_\lambda\geq\frac{1}{\|S_{\lambda'}\|^2}I= \delta^2 I$, which gives 
$\|S_\lambda f\|\geq\delta\|f\|~\mbox{for~ all~} f\in 
l^2(V).$
Therefore, $\sigma_{ap}(S_{\lambda}) \cap \mathbb D_{\delta} = \emptyset.$
It follows that
for 
all $w \in \mathbb{D}_\delta$, $\ker (S_\lambda-w)=\{0\}$ and  
$\mbox{ran}\, (S_\lambda-w)$ is closed. Hence  
$\mbox{ran}\, (S^*_\lambda-w)$ is dense in $\mathscr 
H$ for all 
$w\in\mathbb{D}_\delta$. Since $\mbox{ran}\, (S_\lambda-w)$ is closed, it 
follows 
that $\mbox{ran}\, (S^*_\lambda-\overline{w})$ is closed 
\cite[Chapter XI, Section 6]{Conway}, and hence
$\mbox{ran}\, (S^*_\lambda-\overline{w})=l^2(V)$ for all 
$w\in\mathbb{D}_\delta$. In case $\dim E < \infty$, 
the desired conclusion follows from (iii) and (v) of Theorem \ref{spectral}. 

Suppose now the case in which $\dim E$ is not finite.
Consider the analytic model $({\mathscr 
M}_z, 
\kappa_{\mathscr H},\mathscr H)$  of $S_{\lambda}.$ 
We show that 
$$\{\kappa_{\mathscr H}(\cdot,w)g_i:i=1,\cdots,\ k\}$$
is linearly independent in $\ker (\mathscr M^*_z - \overline{w})$ whenever $\{g_i:i=1,\cdots,\ k\}$ is linearly independent in $E$ for every integer $k \geq 1$ and $w \in \mathbb D_{r_{\lambda}}$. To this end, suppose that 
$\kappa_{\mathscr H}(\cdot,w)g=0$ for some $g\in E$. Then 
$\langle 
U_g(w),g\rangle = \langle U_g,\kappa_{\mathscr H}(\cdot,w)g\rangle=0$. However, by \eqref{Uf}, $U_g=g$ for any $g \in E.$ It follows that $g=0$, and $\dim \ker  
({\mathscr M}_z^*-\overline{w})=\dim E$ for all $w \in \mathbb{D}_{r_{\lambda}}$. 
\end{proof}

%\section{Concluding Remarks}

\section{A model for weighted shifts on rootless directed trees}

In this short section, we show a way to generalize the main result of this paper to the setting of rootless directed trees.
One interest in the theory of weighted shifts on rootless directed trees is due to the fact that these are composition operators in disguise (see \cite[Lemma 4.3.1]{JBS-1}).

We begin with a counter-part of branching index for rootless directed trees.
\begin{definition}
Let $\mathscr T=(V, \mathcal E)$ be a rootless directed tree and let $V_{\prec}$ be the set of branching vertices of $\mathscr T$.
We say that $\mathscr T$ has {\it finite branching index} if there exists
a smallest non-negative integer $m_{\mathscr T}$ such that
$$\childn{k}{V_{\prec}} \cap V_{\prec} = \emptyset~\mbox{for every integer~} k \geq m_{\mathscr T}.$$   
\end{definition}

The role of $\mathsf{root}$ in the notion of the branching index of a rooted directed tree is taken by a special vertex in the context of rootless directed trees with finite branching index as shown below.
\begin{lemma} \label{root}
Let $\mathscr T=(V, \mathcal E)$ be a rootless directed tree with finite branching index $m_{\mathscr T}$.  
Then there exists a vertex $\rot \in V$ such that \beq \label{rot} \mbox{card}(\child{\parentn{k}{\rot}})=1~\mbox{for all integers}~k \geq 1.\eeq Moreover, if $V_{\prec}$ is non-empty then there exists a unique $\rot \in V_{\prec}$ satisfying \eqref{rot}. 
\end{lemma} 
\begin{proof} 
In case $V_{\prec}=\emptyset$, then every vertex of $V$ satisfies \eqref{rot}. Therefore, we may assume that $V_{\prec}$ contains at least one vertex, say, $u_0$.

On contrary, assume that for every $u \in V_{\prec}$ there exists a positive integer $k_u$ (depending on $u$) such that $$\mbox{card}(\child{\parentn{k_u}{u}})=0~\mbox{or~}\mbox{card}(\child{\parentn{k_u}{u}}) \geq 2.$$ Since $\mathscr T$ is rootless, the first case can not occur. Hence $\mbox{card}(\child{\parentn{k_u}{u}}) \geq 2,$ that is, 
$\parentn{k_u}{u} \in V_{\prec}.$
Define inductively $\{u_n\}_{n \geq 0} \subseteq V_{\prec}$ as follows. By assumption, there exists an integer $k_{u_0} \geq 1$ such that $u_1:={\parentn{k_{u_0}}{u_0}} \in V_{\prec}.$
By finite induction, there exist integers $k_{u_1}, \cdots, k_{u_{n-1}} \geq 1$ such that $$u_n:={\parentn{k_{u_0} + k_{u_1} \cdots + k_{u_{n-1}}}{u_0}} \in V_{\prec}.$$
In case $n > m_{\mathscr T},$ $u_0 \in \childn{k_{u_0} + k_{u_1} \cdots + k_{u_{n-1}}}{V_{\prec}} \cap V_{\prec}.$ This is not possible since $\childn{k}{V_{\prec}} \cap V_{\prec} = \emptyset$ for all integers $k \geq m_{\mathscr T}.$ 

%Since we do not need the uniqueness part in the further developement, we simply outline its rather tedious proof. 
To see the uniqueness part, suppose that there exist distinct vertices $\{\rot_i\}_{i=1}^{N}$ in $V_{\prec}$ satisfying \eqref{rot}, where either $N$ is a positive integer bigger than $1$ or $N$ is infinite.
It is easy to see with the help of \eqref{rot} that for integers $i \neq j,$ 
$\parentn{k_1}{\rot_i} \neq \parentn{k_2}{\rot_j}$ for any non-negative integers $k_1$ and $k_2.$ One may now easily verify that $\mathscr T$ has the separation $\mathscr T = \sqcup_{i=1}^N \mathscr T_i,$ where
\beqn
\mathscr T_i = \big(\cup_{ k \geq 1} \childn{k}{\rot_i}\big) \cup  \big\{\parentn{k}{\rot_i} : k \geq 0 \big\}.  
\eeqn
Since $\mathscr T$ is connected, we arrive at a contradiction.
\end{proof}

%The preceding lemma leads to the notion of the generalized root.
%\begin{definition}
%Let $\mathscr T=(V, \mathcal E)$ be a rootless directed tree with finite branching index. 
%The vertex $\rot$, as appearing in the statement of Lemma \ref{root}, will be referred to as the {\it generalized root} of $\mathscr T.$
%\end{definition}

Note that a rootless directed tree $\mathscr T_0$ with empty $V_{\prec}$ is isomorphic to the directed tree with set of vertices $\mathbb Z$ and $\child n = \{n+1\}$ for $n \in \mathbb Z.$ As it is well-known that any weighted shift on $\mathscr T_0$ (to be referred to as {\it bilateral weighted shift}) can be modelled as the operator of multiplication by $z$ on a Hilbert space of formal Laurent series \cite[Proposition 7]{S}, we assume in  
the remaining part of this section that $V_{\prec}$ is non-empty.

We refer to the vertex $\rot \in V_{\prec}$ appearing in the statement of 
Lemma \ref{root} as the {\it generalized root} of $\mathscr T.$
The generalized root may not exist in general. For example,
consider the directed tree $\mathscr T$ with set of vertices
$V=\mathbb Z \times \mathbb Z$ such that
\begin{center}
$\mathsf{Chi}(i,j)=\begin{cases}
                   \{(i,j+1)\}& \text{if}\ j \neq 0,\\
                   \{(i,j+1), (i+1, j)\}& \text{if}\ j=0
                  \end{cases}$
\end{center}
(cf. \cite[Example 4.4]{Ja}).
In this case, $V_{\prec}= \{(i, 0) : i \in \mathbb Z\},$ and hence the set $\child{\parentn{k}{(i, 0)}}$ contains precisely two vertices for any integer $k \geq 1$.

With the notion of generalized root, we immediately obtain the following.
\begin{lemma} \label{rootless}
Let $\mathscr T=(V, \mathcal E)$ be a rootless directed tree with finite branching index $m_{\mathscr T}$ and generalized root $\rot.$
Let $S_{\lambda} \in B(l^2(V))$ be a weighted shift on $\mathscr T.$ 
Let $V^{(2)}:=\{v_k:=\parentn{k}{\rot} : k \geq 1\}$ and let $V^{(1)}:=V \setminus V^{(2)}.$ Let $\mathscr T^{(1)}$ and ${\mathscr T^{(2)}}$ be the directed subtrees corresponding to the sets of vertices $V^{(1)}$ and $V^{(2)}$ respectively.
Then $S_{\lambda}$ admits the following decomposition:
\beq \label{r-less-deco}
S_{\lambda}=\left[\begin{array}{cc}
T_{\lambda} &  \lambda_{\rot} e_{\rot} \otimes e_{v_1} \\
0 & B_{\lambda}     
\end{array}\right] ~\mbox{on~}l^2(V) = l^2(V^{(1)}) \oplus l^2(V^{(2)}),
\eeq
where $T_{\lambda} \in B(l^2(V^{(1)}))$ is a weighted shift on the rooted directed tree $\mathscr T^{(1)}$ with root $\rot$ and finite branching index $k_{\mathscr T^{(1)}}=m_{\mathscr T}$, and $B_{\lambda}  \in B(l^2(V^{(2)})$ is the backward unilateral weighted shift given by \beqn B_{\lambda} e_{v_k}=\begin{cases}
0&~\mbox{if~}k=1\\
\lambda_{v_{k-1}}e_{v_{k-1}}&~\mbox{if~}k \geq 2.
\end{cases}
\eeqn
\end{lemma}
\begin{proof}
Note that $\mathscr T^{(1)}$ is a rooted directed tree with root $\rot.$ Also, 
the set $V_{\prec}$ of branching vertices of $\mathscr T$ is contained in $V^{(1)}$ as
$V^{(2)} \cap V_{\prec} = \emptyset.$
It follows that $k_{\mathscr T^{(1)}}=1+\sup\{n_v : v \in V_{\prec}\}.$ Since $m_{\mathscr T}$ is the smallest integer such that 
$\childn{m_{\mathscr T}}{V_{\prec}} \cap V_{\prec} = \emptyset,$ we must have $\sup\{n_v : v \in V_{\prec}\}=m_{\mathscr T}-1.$
This shows that $\mathscr T^{(1)}$ has branching index precisely $m_{\mathscr T}$.

Since $S^*_{\lambda}e_{v_k} = \lambda_{v_k}e_{v_{k+1}}$,
$l^2(V^{(2)})$ is invariant under $S^*_{\lambda}$. This gives us the decomposition \beqn
S_{\lambda}=\left[\begin{array}{cc}
S_{\lambda}|_{l^2(V^{(1)})} &  P_1S_{\lambda}|_{l^2(V^{(2)})}\\
0 & P_{2}S_{\lambda}|_{l^2(V^{(2)})}     
\end{array}\right] ~\mbox{on~}l^2(V) = l^2(V^{(1)}) \oplus l^2(V^{(2)}),
\eeqn
where $P_i$ denotes the orthogonal projection of $l^2(V)$ onto ${l^2(V^{(i)})}$ for $i=1, 2.$
It is easy to see that $P_{1}S_{\lambda}|_{l^2(V^{(2)})}$ is the rank one operator $\lambda_{\rot} e_{\rot} \otimes e_{v_1}.$ That $P_{2}S_{\lambda}|_{l^2(V^{(2)})} =B_{\lambda}$ is also a routine verification. 
%This completes the proof of the lemma.
\end{proof}
\begin{remark}
Note that every weighted shift on a rootless directed tree with finite branching index is an extension of a weighted shift on a rooted direct tree with finite branching index.
\end{remark}

We illustrate the result above with the help of the following simple example.

\begin{example}
Consider the directed tree $\mathscr T$ with set of vertices 
$$V:=\{(1,i), (2,i):i\geq1\} \cup \{-k : k \geq 0\}.$$ We further require that $
\mathsf{Chi}(-k) = -(k-1)$ if $k \geq 1,$
$\mathsf{Chi}(0)=\{(1,1),(2,1)\}$ and 
\[\mathsf{Chi}(1,i)=\{(1,i+1)\},\
\mathsf{Chi}(2,i)=\{(2,i+1)\},\ \text{for all}\ i\geq1. \]
In this case, the branching index $m_{\mathscr T}=1$ and the generalized root $\rot$ is $0.$ Also, $V^{(1)} =\{(1,i), (2,i):i\geq1\} \cup \{0\}$ and $V^{(2)} = \{-k : k \geq 1\}.$ 
The weighted shift $T_{\lambda}$ on $\mathscr T^{(1)}$ as defined in the last lemma can be identified with the weighted shift on the directed tree $\mathscr T_2$ (with root $0$) as discussed in Example \ref{tridiagonal}. Further, the rank one operator $\lambda_{\rot} e_{\rot} \otimes e_{v_1}$ is precisely
$\lambda_0 e_0 \otimes e_{-1}$. Finally,
the backward unilateral weighted shift $B_{\lambda}$ can be identified with the adjoint of the weighted shift on the directed tree $\mathscr T_1$ (with root $-1$) as discussed in Example \ref{diagonal}.
\end{example}

%We now present an analytic model for a left-invertible weighted shift on a rootless directed tree with finite branching index.
We now present a counter-part of Theorem \ref{thm1} for rootless directed trees.
\begin{theorem} \label{deco-rootless}
Let $\mathscr T=(V, \mathcal E)$ be a rootless directed tree with finite branching index and generalized root $\rot.$
Let $S_{\lambda} \in B(l^2(V))$ be a left-invertible weighted shift on $\mathscr T.$ Then there exist a Hilbert space $\mathscr H$ of vector-valued Holomorphic functions in $z$ defined on a disc in $\mathbb C$, and a Hilbert space $\mathcal H$ of scalar-valued holomorphic functions in $t$ defined on a disc in $\mathbb C$ such that $S_{\lambda}$ is unitarily equivalent to
\beqn
\left[\begin{array}{cc}
{\mathscr M}_z &  f \otimes g \\
0 & M^*_t     
\end{array}\right] ~\mbox{on~}  \mathscr H \oplus \mathcal H,
\eeqn
where ${\mathscr M}_z$ is the operator of multiplication by $z$ on $\mathscr H,$ $f \otimes g$ is a rank one operator with
$f \in \ker \mathscr M^*_z \setminus \{0\},$ $g \in \ker M^*_t \setminus \{0\},$ and
$M_t$ is the operator
of multiplication by the co-ordinate function $t$ on $\mathcal H.$    
\end{theorem}
\begin{proof}
By Lemma \ref{rootless}, $S_{\lambda}$ admits the decomposition \eqref{r-less-deco}. Since $S_{\lambda}$ is left-invertible, so are $T_{\lambda}$ and $B^*_{\lambda}.$ The desired decomposition
now follows immediately from Theorem \ref{thm1}.
\end{proof}
\begin{remark}
A routine calculation shows that the self-commutator $[S^*_{\lambda}, S_{\lambda}]:=S^*_{\lambda}S_{\lambda} - S_{\lambda}S^*_{\lambda}$ of $S_{\lambda}$ (upto unitary equivalence) is equal to 
\beqn
\left[\begin{array}{cc}
[{\mathscr M}^*_z, \mathscr M_z] - f\otimes f &  0 \\
0 & g\otimes g - [M^*_t, M_t]     
\end{array}\right].
\eeqn
In particular, $[S^*_{\lambda}, S_{\lambda}]$ is compact if and only if so are $[\mathscr M^*_z, \mathscr M_z]$ and $[M^*_t, M_t].$ 
\end{remark}

We conclude this paper with one application to the spectral theory of weighted shifts on rootless directed trees (cf. \cite[Theorem 2.3]{C}).

\begin{corollary} With the hypotheses and notations of Theorem \ref{deco-rootless}, we have \beqn \sigma_e(S_{\lambda}) = \sigma_e(\mathscr M_z) \cup \sigma_e(M^*_t).\eeqn
If, in addition, $S_{\lambda}$ is Fredholm then so are $\mathscr M_z$ and $M_t.$ In this case, \beqn \mbox{ind}\,S_{\lambda} = \mbox{ind}\,\mathscr M_z + 1.\eeqn
\end{corollary}

%\section{Main Result} 
%
%
%\begin{proposition}\label{ABS}
%Let $S_{\lambda} \in B(l^2(V))$ be a weighted shift on a rooted directed tree $\mathscr T$. If $\ker S^{*k}_\lambda$ is of finite dimension then $S^{*}_\lambda$ is an abstract backward shift. 
%\end{proposition}
%
%\begin{proof}
%\end{proof}

%\begin{acknowledgements}\label{ackref}
\medskip \textit{Acknowledgment}. \
We express our sincere thanks to Jan Stochel and Zenon Jan Jab{\l}o\'nski for many helpful suggestions.
In particular, we acknowledge drawing our attention to the work \cite{C} on the spectral theory of composition operators. Further, the first author is thankful to the faculty and the administrative unit of School of Mathematics,
Harish-Chandra Research Institute, Allahabad for their warm hospitality during the preparation of this paper.
%\end{acknowledgements}

\end{document}